\documentclass[12pt, twoside]{amsart}

\usepackage{amssymb,amsfonts,amstext,amsthm,bm,amsmath}
\usepackage{amscd,amstext,mathrsfs,euscript,a4wide}  
\usepackage{booktabs, mathdots, accents}
\usepackage[british]{babel}
\usepackage{graphicx}
\usepackage[hidelinks]{hyperref} 
\usepackage{url}
\usepackage[all,cmtip]{xy} 
\usepackage{longtable,multicol,multirow,array}
\usepackage{enumitem}

\normalsize
\def\D{\textup{d}} 
\def\z{\mathfrak{z}}
\def\d{\mathfrak{d}}

\def\k{\mathfrak{k}}

\def\g{\mathfrak{g}}
\def\h{\mathfrak{h}}
\def\b{\mathfrak{b}}
\def\n{\mathfrak{n}}

\def\a{\mathfrak{a}}
\def\q{\mathfrak{q}}
\def\s{\mathfrak s}

\def\r{\mathfrak{r}}

\def\R{\mathbb{R}}

\def\Z{\mathbb{Z}}
\def\N{\mathbb{N}}
\def\ad{\operatorname{ad}}
\DeclareMathOperator{\End}{\operatorname{End}\,}
\DeclareMathOperator{\Id}{\operatorname{Id}}
\renewcommand{\Im}{\operatorname{Im}}
\def\Ker{\operatorname{Ker}}
\renewcommand{\geq}{\geqslant}\renewcommand{\leq}{\leqslant}
\newcommand{\lie}[1]{\mathfrak{#1}}
\DeclareMathOperator{\SL}{\operatorname{SL}}
\DeclareMathOperator{\SU}{\operatorname{SU}}
\DeclareMathOperator{\U}{\operatorname{U}}
\DeclareMathOperator{\slie}{\lie{sl}}
\DeclareMathOperator{\so}{\lie{so}}
\DeclareMathOperator{\su}{\lie{su}}
\def\aff{\mathfrak{aff}}
\newcommand{\bigslant}[2]{{\raisebox{.2em}{$#1$}\left/\raisebox{-.2em}{$#2$}\right.}} 
\def\pint{\langle \cdotp,\cdotp \rangle }
\def\ric{\operatorname{Ric}}
\def\rad{\operatorname{rad}}
\def\tr{\operatorname{tr}}
\newcommand{\iso}{\cong}
\newcommand{\then}{\Longrightarrow}
\newcommand{\op}[1]{\operatorname{#1}}
\newcommand{\restr}{\raisebox{-.2em}{$\!\big\arrowvert$}}
\newenvironment{spmatrix}{\left(\smallmatrix}{\endsmallmatrix\right)}
\newtheorem{definition}{\bf Definition}[section]
\newtheorem{prop}[definition]{Proposition}
\newtheorem{coro}[definition]{Corollary}
\newtheorem{lemma}[definition]{Lemma}
\newtheorem{rem}[definition]{Remark}

\newtheorem{thm}[definition]{Theorem}

\newtheorem{example}[definition]{\bf Example}
\newtheorem{examples}[definition]{\bf Examples}
\theoremstyle{definition}
\newtheorem*{acknowledgements}{Acknowledgements}

\begin{document}

\title{$\boldsymbol \eta$-Einstein Sasakian Lie algebras}
\date{\today}

\subjclass[2000]{}

\author[Andrada]{Adri{\'a}n  Andrada}
\address[\href{mailto:adrian.andrada@unc.edu.ar}{AA}]{FaMAF, Universidad Nacional de C\'ordoba and Ciem-Conicet, Av. Medina Allende s/n, Ciudad Universitaria, X5000HUA C\'ordoba (AR)}

\author[Chiossi]{Simon G. Chiossi}
\address[\href{mailto:simongc@id.uff.br}{SGC}]{
Departamento de Matem{\'a}tica Aplicada, Instituto de Matem{\'a}tica e 
Estat{\'i}stica, Universidade Federal Fluminense, Rua professor Marcos Waldemar de Freitas Reis s/n (bloco G), 24210-201 Niter{\'o}i/RJ (BR)}

\author[N\~unez]{Alberth J. N{\'u}\~nez Sullca}
\address[\href{mailto:alberthnunez@id.uff.br}{AN}]{
PGMat, Instituto de Matem{\'a}tica e Estat{\'i}stica, Universidade Federal Fluminense, Rua professor Marcos Waldemar de Freitas Reis s/n (bloco H), 24210-201 Niter{\'o}i/RJ (BR)}

\frenchspacing
\thispagestyle{empty}

\begin{abstract}
We study $\eta$-Einstein Sasakian structures on Lie algebras, that is, Sasakian structures whose associated Ricci tensor satisfies an Einstein-like condition. We  divide into the cases in which the Lie algebra's centre is non-trivial (and necessarily one-dimensional) from those where it is zero. In the former case we show that any Sasakian structure on a unimodular Lie algebra is $\eta$-Einstein. As for centreless Sasakian Lie algebras, we devise a complete characterisation under certain dimensional assumptions regarding the action of the Reeb vector. Using this result, together with the theory of normal $j$-algebras and modifications of Hermitian Lie algebras, we construct new examples of $\eta$-Einstein Sasakian Lie algebras and solvmanifolds, and provide effective  restrictions for their existence. 
\end{abstract}

\maketitle

\begin{center}{\it Dedicated to our dear Marisa Fern{\'a}ndez}\end{center}

\tableofcontents


\section{Introduction}

A Sasakian manifold is the odd-dimensional counterpart of a K{\"a}hler manifold, due to the fact that a Riemannian manifold $(M, g)$ of dimension $2n+1$ is Sasakian if and only if its Riemannian cone $M\times \R^{+}$ is 
K{\"a}hler. Sasakian structures on $(M, g)$ are triples $(\varphi, \eta ,\xi)$, consisting of an endomorphism $\varphi\in\so(TM)$, a vector field $\xi\in TM$ (the Reeb vector) and the metric-dual form $\eta\in TM^*$, satisfying certain compatibility conditions (\S \ref{sec: sasakian}). These manifolds have been  studied at length since Sasaki introduced them $60$-odd years ago. 
\smallbreak 
A distinguished family of Sasakian manifolds are those whose metric is Einstein, meaning that the associated Ricci tensor satisfies $\ric=\lambda g$ for some $\lambda \in \R$. They are naturally called Sasaki-Einstein manifolds, and enjoy remarkable properties. For instance, it turns out that $\lambda=2n$, so $g$ has positive Ricci curvature, whence $\pi_1(M)$ is finite in case $g$ is complete, by Myers' theorem. The prototypical examples are the spheres $S^{2n+1}$ with the standard K{\"a}hler metric. The  textbooks  \cite{Blair} and \cite{Boyer-Kris-BOOK} have become classical references on Sasakian geometry. 
\smallbreak 
 Okumura \cite{Oku} defined a generalisation of Sasaki-Einstein manifolds by imposing a condition that relates $\ric$ not only to $g$, but also to the Reeb $1$-form $\eta$. A Sasakian manifold is called $\eta$-Einstein if $\ric=\lambda g+\nu \,\eta \otimes \eta$ for some constants $\lambda, \nu \in \R$. The Sasakian-Einstein case is recovered for $\nu=0$. The constants $\lambda$ and $\nu$ are not arbitrary as they must obey $\lambda+\nu=2n$, and it can be shown that demanding $g$ be $\eta$-Einstein is equivalent to  saying the metric is Einstein on the distribution transversal to the Reeb vector. 
 From that, every Sasakian $\eta$-Einstein manifold has constant scalar curvature equal to $2n(\lambda+ 1)$. 
It was shown in \cite{BGM} that the value $\lambda=-2$ plays an important role in this context, for such manifolds can now be labelled positive (if $\lambda > -2$), null ($\lambda=-2$) and negative ($\lambda< -2$). In particular, therefore, Sasaki-Einstein manifolds are positive. \smallbreak 
According to \cite{Geiges}, a $3$-dimensional compact Sasakian manifold is diffeomorphic to the quotient of either $S^{3}\iso \SU(2)$, the Heisenberg group $H_3$ or the universal cover $\widetilde{\SL}(2,\R)$ by some lattice (a co-compact discrete subgroup). 
These three model geometries correspond precisely to positive, null or negative Sasakian $\eta$-Einstein structures. That  the universal cover of any of the above is a Lie group motivates the study of ($\eta$-Einstein) Sasakian structures on Lie groups. The investigation of Lie groups equipped with left-invariant Sasakian structures was initiated in \cite{AFV}. That article explains why the Lie algebras of Sasakian Lie groups can be divided into two classes: the ones with one-dimensional centre (which has to be spanned by the Reeb vector), and the ones with trivial centre. The first class is well understood, for it consists of central extensions of K{\"a}hler Lie algebras by the cocycle given by the K{\"a}hler form. Much less is known regarding Sasakian Lie algebras with trivial centre. The same paper provides us with a classification of $5$-dimensional Sasakian Lie algebras, and highlights the $\eta$-Einstein ones among them: among the $5$-dimensional Lie algebras admitting an $\eta$-Einstein Sasakian structure are the Heisenberg algebra $\h_5$, a very specific non-unimodular solvable Lie algebra 
and the product $\mathfrak{sl}(2,\R)\times \mathfrak{aff}(\R)$. 
As it turns out, these structures are all negative or null $\eta$-Einstein.\smallbreak
The present article's goal is to study higher-dimensional Lie groups equipped with a left-invariant $\eta$-Einstein Sasakian structure. To do so the first thing is to split the two cases: non-trivial centre and trivial centre. Over the former class the analysis is simpler, and we are able to show (Proposition \ref{prop: einstein con centro}) that a Sasakian Lie algebra with non-trivial centre is $\eta$-Einstein (with Einstein constant $\lambda$) precisely if it is the central extension of a K{\"a}hler-Einstein Lie algebra (with defining constant $\lambda+2$). Moreover, Corollary \ref{coro: eta-einstein flat} proves that if a Sasakian Lie algebra with non-trivial centre is unimodular then the structure is automatically $\eta$-Einstein, and actually null ($\lambda=-2$). In particular, Sasakian solvmanifolds whose Lie algebra has non-trivial centre must be null $\eta$-Einstein, and we provide in Example \ref{ex: null solvmanifold} original examples of solvmanifolds equipped with null $\eta$-Einstein Sasakian geometry. \par
\quad At the opposite end, if the centre is trivial, the Reeb vector $\xi$ no longer commutes with everything, but we can use it to split the Lie algebra $\g$ into an  orthogonal sum of the kernel of $\ad_\xi$ (a Sasakian subalgebra with non-trivial centre) and the image of $\ad_\xi$. Alas, the computations in this setting are considerably more involved and give little insight per se. We therefore restrict to the case where the image of $\ad_\xi$ is two-dimensional, which at least geometrically is reasonable. This extra assumption enables us to characterise all centreless Sasakian Lie algebras of this type (Proposition \ref{prop: compatibilidad} and Theorem \ref{carac sasaki}). Our characterisation involves a K{\"a}hler Lie algebra $\h_1$ (with exact K{\"a}hler form) and a distinguished vector $H_0$ in the commutator ideal of $\h_1$. As it turns out, the element $H_0$ plays a key  role in the study. For instance, if $H_0=0$ it follows that $\g$ is never solvable, whereas if $H_0\neq 0$ and $\h_1$ is solvable then $\g$ must be solvable as well. 
Regarding the $\eta$-Einstein condition, we show in Proposition \ref{H=0} that under the assumption $H_0=0$, the Sasakian Lie algebra $\g$ is $\eta$-Einstein if and only if $\h_1$ is K{\"a}hler-Einstein. The generic case ($H_0\neq 0$) is more difficult to address, and to tackle it we suppose $\h_1$ is solvable. 
Since semisimple Lie algebras cannot carry K{\"a}hler forms (nor any symplectic form, by the way), this further hypothesis is rather natural in the light of the Levi decomposition. The same assumption becomes even more significant due to the non-accidental appearance of so-called normal $j$-algebras (or $J$-algebras according to \cite{Dalek-Keizo-Kamishima}). These were introduced by Piatetskii-Shapiro and much studied by, first and foremost, Dorfmeister, D'Atri and Dotti. After a careful description, we use them in Theorem \ref{thm: examples with H0 no nulo} to construct new examples of $\eta$-Einstein Sasakian Lie algebras of negative type. 
\smallbreak 
Subsequently, we bring in the concept of modification of a metric Lie algebra with a two-fold purpose: first, we show in Theorem \ref{thm: lattices} that if $G$ is an arbitrary Lie group of dimension $\geq 5$ that admits a left-invariant $\eta$-Einstein Sasakian structure and contains a lattice, then $G$ is solvable and the $\eta$-Einstein structure is null. Secondly, we prove that any K{\"a}hler-exact Lie algebra is the modification of a normal $j$-algebra. We then employ this fact for showing that there are indeed only finitely many possibilities to pick $H_0$ 
(Proposition \ref{prop: finite}), based on which we also provide obstructions to the existence of $\eta$-Einstein structures (Proposition \ref{prop: restriction}).  
\smallbreak 
 Finally, in Proposition \ref{prop: final} we consider the extreme case when the image of $\ad_{\xi}$ is largest possible. We prove that this can only happen in dimension $3$, so the algebras of the case are $\mathfrak{su}(2)$ or $\mathfrak{sl}(2,\R)$.

\begin{acknowledgements}
This work is part of Nu\~nez' PhD thesis, and preliminary results were presented at the  2024 conference `Special holonomy and geometric structures on complex manifolds' held at {\sc impa}. He thanks the PhD programme at {\sc uff}, and {\sc capes}, for the financial support during his graduate studies and for the PDSE grant 6/2024. \par
\quad AN and SGC are grateful to {\sc famaf} (Facultad de Matem{\'a}tica, Astronom{\'i}a, F{\'i}sica y Computaci{\'o}n at Universidad Nacional de C{\'o}rdoba) and {\sc ciem} (Centro de Investigaci{\'o}n y Estudios de Ma\-te\-m{\'a}\-tica of {\sc conicet}) for the unwavering warm hospitality in C{\'o}rdoba. 
The three authors are indebted to Anna Fino, Gueo Grantcharov, Giulia Dileo, Ilka Agricola and Jorge Lauret for insightful conversations and for pointing out useful references. They thank the referee for reading the paper's first version carefully and making spot-on comments and improvements.\par
SGC is a member of {\sc in}d{\sc am} (Istituto Nazionale di Alta Matematica `Francesco Severi'). 
\end{acknowledgements}

\bigbreak

\section{Preliminaries}

\subsection{Sasakian manifolds}\label{sec: sasakian}

We are interested in Sasakian Lie algebras and their curvature features, with an eye to $\eta$-Einstein metrics. For completeness we shall begin by laying out the ambient framework, the standard notation and swiftly review a few known results.
\smallbreak

For starters, an \textit{almost contact structure} on an odd-dimensional smooth real manifold $M$ is a triple $(\varphi,\eta,\xi)$ where $\xi$ is a nowhere-vanishing vector field (called the \textit{Reeb vector field}), $\eta$ is a $1$-form satisfying $\eta(\xi)=1$ and $\varphi$ is an endomorphism of $TM$ such that 
$$ 
 \varphi^{2}=-\Id +\xi \otimes \eta.
$$
One refers to $(M,\varphi,\eta,\xi)$ as an \textit{almost contact manifold}. 
The above conditions 
force $\varphi(\xi)=0$ and $\eta \circ\varphi=0$, and as a consequence  the tangent bundle pointwise splits into a horizontal hyperplane plus a line
\[TM=\Ker \eta\oplus \R\xi,\]
and furthermore $\Ker \eta=\Im\varphi$. Associated with $\varphi$ is the  Nijenhuis $(2,1)$-tensor $N_{\varphi}$:
\begin{equation}\label{Nijenhuis definicion}
N_{\varphi}(X,Y):=\varphi^{2}[X,Y]+[\varphi X,\varphi Y]-\varphi[\varphi X,Y]-\varphi[X,\varphi Y],
\end{equation}
and the almost contact structure $(\varphi, \eta, \xi)$ is called $\textit{normal}$ if 
$$
    N_{\varphi}=-\D\eta \otimes \xi.
$$

The latter is equivalent \cite{Sasaki} to requiring that the almost complex structure
\begin{equation}\label{J en cono}
    J\left(X, f\tfrac{d}{d t}\right)= \left(\varphi X-f\xi,\eta(X)\tfrac{d}{d t}\right)
\end{equation}
 on $M\times \R$ be integrable, where $f\in\mathcal{C}^\infty(M \times \R)$ and $t$ is a coordinate on $\R$.
\medbreak

One says a  Riemannian metric $g$ on an almost contact manifold $(M,\varphi,\eta,\xi)$ is compatible with the almost contact structure if 
    \[ g(\varphi X,\varphi Y)=g(X,Y) -\eta(X)\eta(Y)\]
    for any $X,Y\in TM$. This defines a reduction of the manifold's structure group to $\U(n)$ and $(\varphi,\eta,\xi, g)$ becomes known as an \textit{almost contact metric structure}.
In such a quadruple the tensor $\varphi$ is skew-symmetric for $g$, and the Reeb field is the metric dual of $\eta$. Having a metric and $\varphi$ enables one to define the so-called fundamental $2$-form $\Phi\in \bigwedge^2T^*M$ by 
\[\Phi(X,Y):=g(X, \varphi Y),\]
in perfect analogy to the almost Hermitian setting. The classical Chinea--Gonz{\'a}lez classification of types of almost contact metric structure based on the intrinsic torsion has finally found a very geometric interpretation and reworking in \cite{Ilka}.
\medskip

As $M$ will typically be $(2n+1)$-dimensional, we will write $M^{2n+1}$ to remind ourselves.

\begin{definition}
    An almost contact metric structure $(\varphi,\eta,\xi, g)$ on a manifold $M^{2n+1}$ is called {\bf Sasakian} provided it is normal and
$\D \eta=2\Phi$ 
(making $\Phi$ exact). The data $(M,\varphi,\eta,\xi, g)$ is called a {\bf Sasakian manifold}.
\end{definition}
The relation $\D\eta =2\Phi$ implies $\eta$ is a contact form (ie, $\eta \wedge (d\eta)^n$ is a volume form).
\medbreak

There are a number of reasons for which Sasakian manifolds represent the most significant class of almost contact metric manifolds. First and foremost is their strong link to K{\"a}hler manifolds: the Riemannian cone of an almost contact metric manifold $M$, equipped with the structure induced by \eqref{J en cono}, is K{\"a}hler if and only if $M$ is Sasakian (see \cite{Boyer}, but also \cite{Grantcharov-Ornea}).
\smallbreak

As Sasakian manifolds are central to the present work, we shall 
recall a few essential properties that will be useful in upcoming sections. They are standard and proofs can be found in \cite{Blair,Boyer-Kris-BOOK}.
\medbreak

Henceforth $\nabla$ will indicate the Levi-Civita connection associated with $g$, and $R$ and $\ric$ will be the Riemannian and Ricci tensors. Our convention for the Riemann curvature is $R_{x,y}:=[\nabla_{x},\nabla_{y}] - \nabla_{[x,y]}$.

\begin{prop}\label{nabla xi}
For any vector fields $X,Y$ on a Sasakian manifold  $(M^{2n+1},\varphi,\eta,\xi, g)$ we have:
    \begin{enumerate}
        \item $\nabla_{X}\xi= - \varphi X$\; and\; $\nabla_{\xi}X=[\xi,X]-\varphi X$,\smallskip
        \item $(\nabla_{X}\varphi)Y=g(X,Y) \xi-\eta(Y)X$,\smallskip
        \item $R_{X,Y}\xi=\eta(Y)X-\eta(X)Y$, \smallskip
        \item $\ric(\xi,X)=2n\,\eta(X)$.
    \end{enumerate}
\end{prop}

In particular, the Reeb vector is a Killing vector field ($\nabla_\xi \xi=0$) 
and $\ric(\xi,\xi)=2n$, so a Sasakian metric is never Ricci-flat.

\medskip

Our true focus is a distinguished family of Sasakian manifolds, called  $\eta$-Einstein because they are defined in terms of the Ricci tensor just like ordinary Einstein manifolds are. They were introduced by Okumura \cite{Oku} (see also \cite{sas}), and a good starting source of study is \cite{BGM}.

\begin{definition}\label{def:eta-ES} 
A Sasakian manifold $(M^{2n+1},\varphi,\xi,\eta, g)$ is said to be {\bf $\boldsymbol \eta$-Einstein} if the Ricci tensor satisfies the constraint 
$$
\ric = \lambda g + (2n-\lambda) \eta \otimes \eta
$$
for some constant $\lambda\in \R$. 
More specifically, it is called {\bf positive} ($\eta$-Einstein) if $\lambda >-2$, {\bf null} if $\lambda=-2$, and {\bf negative} if $\lambda < -2$.\par
\end{definition}

Observe that the $\eta$-Einstein condition reads 
$$
    \ric(U,V)=\lambda g(U,V)
$$
for all Reeb-transversal vector fields $U,V$. 
Another immediate consequence of the definition is that every $\eta$-Einstein  Sasakian manifold automatically has constant scalar curvature $2n(\lambda+1)$.

\subsection{Sasakian Lie algebras}

We shall study left-invariant structures on Lie groups $G$. 
As these are uniquely determined by their value at the group's identity,  we are entitled to consider structures on Lie algebras $\g=T_eG$. Just for the purpose of collecting all the pieces of data presented earlier, let us recall that a metric Lie algebra $(\g,\pint)$ is a Lie algebra $\g$ equipped with a positive-definite, symmetric bilinear form $\pint$.
Any such inner product corresponds uniquely to a left-invariant Riemannian metric on each connected group $G$ integrating $\g$.\par
\quad  The linear map $\D:\g^* \longrightarrow \bigwedge^2 \g^*$ that sends a $1$-form 
 $\alpha$ to $\D \alpha$, 
 $$\D \alpha(x,y) := -\alpha [x,y],$$          
is the (Lie algebra) exterior derivative or Chevalley-Eilenberg differential. It extends to a graded linear map $\D:\bigwedge^k \g^* \longrightarrow \bigwedge^{k+1} \g^*$ in any degree 
which, spelt out, reads 
\[(\D\phi)(x_0, x_1, \ldots,  x_k)=\sum_{0\leq i < j \leq k}(-1)^{i+j}\phi([x_i, x_j],x_0,\ldots, \widehat{x_i}, \dots , \widehat{x_j}, \dots ,x_{k})\]
where the hat means that the term is omitted. \par
\quad If $(e_1,\ldots, e_n)$ is a basis of $\g$, we will denote by $(e^1,\ldots, e^n)$ the dual basis of $\g^*$ and conventionally write $e^{ij}:=e^{i} \wedge e^{j}$. 
Then the string 
\[(\D e^1,\D e^2,\ldots, \D e^n)\]
encodes the structure equations $[e_i, e_j]=\sum_{k}c_{ij}^ke_k$ since $\D e^k=-\sum_{i<j} c_{ij}^k e^{ij}$. This is very practical, for the Jacobi identity on $\g$ is equivalent to asking the exterior algebra $(\bigwedge \g^*, \D)$ be a complex.\smallskip

Two major players in our work are the $2$-dimensional affine Lie algebra 
and the $(2n+1)$-dimensional Heisenberg algebra, so we recall them for the record:
\begin{itemize}
\item[]\hspace{-1cm} $\aff(\R)$ is spanned by vectors $e_1,e_2$ with structure equations $(0, e^{21})$;\smallbreak 
\item[]\hspace{-1cm} $\h_{2n+1}=\operatorname{span}(e_1, \ldots, e_n, f_1, \ldots, f_n, z)$ with 
equations $(0,\ldots, 0, f^1\wedge e^1+\cdots + f^n\wedge e^n)$.
\end{itemize}

\begin{definition} A {\bf Sasakian Lie algebra} consists of a metric Lie algebra $\big(\g, \pint\big)$, a vector $\xi\in\g$, its metric dual $\eta \in \g^*$ and an endomorphism $\varphi \in \g^*\otimes\g$ subject to:
$$\begin{array}{cc}
 \varphi^{2}=-\Id +\xi\otimes \eta, \qquad
\langle \varphi x, \varphi y\rangle= \langle x, y \rangle - \eta(x)\eta(y)\\[2mm]
\D \eta(x,y) = 2 \langle x,\varphi y\rangle, \qquad N_{\varphi}=-d\eta \otimes \xi,
\end{array}$$
with $N_\varphi$ as per \eqref{Nijenhuis definicion}. 
\end{definition}
As before, we set $\Phi(x,y):=\langle x,\varphi y\rangle$. Now, the normality  condition is equivalent (see \cite[Proposition 3.1]{AD}) to the pair of requirements 
$$\ad_{\xi}\circ \varphi=\varphi \circ \ad_{\xi}, \qquad
[\varphi x, \varphi y]-[x,y]=\varphi\big([\varphi x,y]+[x,\varphi y]\big)$$
for all $x,y \in \Ker \eta$.
\medbreak

Notation-wise, let us name the nullspace of $\eta$ and the projection of the Lie bracket to this nullspace by 
$$ \h:= \Ker\eta, \qquad \theta: \h\times \h\to \h, \;\; \theta(\cdot,\cdot):= \operatorname{pr}_\h [\cdot,\cdot]\restr_{\h\times \h},
$$
where $\operatorname{pr}_\h:\g\to\h$ denotes the projection coming from the decomposition $\g=\R\xi\oplus \h$. 

The next two results, proved in \cite{AFV}, contain the starting point of our investigations. First, the centre $\z$ of a Sasakian Lie algebra is either trivial or  spanned by $\xi$ [cit, proposition 3.4].
We shall say that $\g$ has a centre or has no centre/is centreless in order to distinguish the two situations. Secondly, if the Reeb vector is central then $\theta$ is a Lie bracket on $\h$, and we mod $\xi$ out so that the algebra $(\h\iso \g/\R, \theta$) acquires a K{\"a}hler structure
$$ \pint_{\h}:=\pint\restr_{\h\times\h},\qquad J:=\varphi \restr_\h, \qquad \omega:= \Phi\restr_{\h\times\h}$$ 
induced from the geometry of $\g$. 
The same recipe, read backwards, prescribes how the central extension of an arbitrary K{\"a}hler Lie algebra (with K{\"a}hler form used as cocycle) is turned into a Sasakian algebra \cite{AFV}.
A $3$-dimensional Sasakian Lie algebra is isomorphic to either $\su(2)$ or $\slie(2, \R)$ if simple, or $\aff(\R) \times \R$ or the Heisenberg algebra $\h_3$ if it has centre \cite{Geiges, Cho-Chun}. The concern then starts with $\dim\g=5$. In that dimension a centreless Sasakian Lie algebra must be $\su(2) \times \aff(\R)$,\ $\slie(2, \R)\times \aff(\R)$, or a non-unimodular solvable algebra of the form $\R^2\ltimes \h_3$ \cite{AFV}. \par
\quad Under the assumption that $\g$ has non-trivial centre, by the way, Cartan's criterion ensures that solvability is preserved back and forth: $\g$ solvable $\iff$ $\h$ solvable.
\medbreak 

We recall that if a simply connected Lie group admits lattices then it must be unimodular. Being able to check unimodularity using the trace of the adjoint representation is clearly much better than doing so at the group's level, if only because the Haar measure might not be known explicitly. Nilpotent, reductive and compact Lie algebras/groups are unimodular, but solvable ones may not ($\aff(\R)$ being the simplest instance). Finding simply connected solvable Lie groups admitting lattices is not an easy task, and few general criteria are known to determine whether a given solvable Lie group admits lattices: one such regards nilpotent Lie groups and is due to Malcev, while Bock \cite{Bo} addresses   interesting non-nilpotent classes. Others, thought less practical ones, can be found in the encyclopedic compendium \cite[Ch.2, \S 3.7]{VGS}. We just mention that the presence of a lattice in a solvable Lie group forces so-called strong unimodularity \cite{Garland}. 
The classification of $4$-dimensional unimodular Lie groups, with and without lattices, can be found for instance in \cite{Keizo-Kamishima}. 
We also recall that geometric structures on solvable Lie groups are interesting even if they do not give rise to compact examples. For instance, $1$-connected Riemannian homogeneous spaces of non-positive curvature can be directly regarded as solvable groups $G$ with an appropriate left-invariant metric, and much is known in the unimodular case \cite{Mil, AW1,AW2}. 
Specifically for Sasakian Lie groups, the reference articles are  
\cite{Dalek-Vicente-Keizo-Kamishima, Dalek-Keizo-Kamishima, Cortes}.

\section{Einstein Sasakian algebras with centre}\label{sec:ES with centre} 

We wish to study the curvature properties of Sasakian algebras $\g$. Suppose $\xi$ is central, so that the bracket on 
\[\g=\R\xi \oplus \h\]
boils down to 
\begin{equation}\label{corchete de extension}
    [x,y]=-2\omega(x,y)\xi+\theta(x,y) 
\end{equation}
on the K{\"a}hler reduction $(\h, \theta, J, \pint_\h, \omega)$.  
\medbreak

In the sequel the sub-/super-script $^\h$ will denote any object (connection, tensor and the like) calculated on $\h$. The Riemannian connection and the  curvature tensors will be defined in terms of $\pint_\h$.
 
\begin{lemma}\label{Curvatura de g y h}
Let $\g$ be Sasakian with centre and $\h$ its K{\"a}hler quotient by $\xi$. The respective curvatures are  related by 
 \[R_{x,y}z=R^{\h}_{x,y}z+\omega(y, z)\varphi(x)-\omega(x,z)\varphi(y)-2\omega(x,y)\varphi(z)\]
    for all $x,y, z \in \h$. 
\end{lemma}

\begin{proof}
Take any $x,y, z \in \h$. Implementing the formulas 
   \begin{equation}\label{covariante en g y h}
      \nabla_{x}y=-\omega(x,y)\xi+ \nabla^{\h}_{x}y, \quad 
      \nabla_{\xi}x=-\varphi(x) 
   \end{equation}
lifted from \cite[Section 2]{to}, we obtain 
$$
    \nabla_{x}\nabla_{y}z=\nabla_{x}(-\omega(y, z)\xi + \nabla^{\h}_{y}z)=\omega(y, z)\varphi(x)-\omega(x,\nabla^{\h}_{y}z)\xi+\nabla^{\h}_{x}\nabla^{\h}_{y}z.
$$

Using \eqref{corchete de extension},  
$$
\nabla_{[x,y]}z=\nabla_{-2\omega(x,y)\xi+\theta(x,y)}z=2\omega(x,y)\varphi(z)-\omega(\theta(x,y),z)\xi+\nabla^{\h}_{\theta(x,y)}z,
$$
so eventually: 
\begin{equation*}\begin{split}
    R_{x,y}z &=\omega(y, z)\varphi(x)-\omega(x,\nabla^{\h}_{y}z)\xi+\nabla^{\h}_{x}\nabla^{\h}_{y}z - \omega(x,z)\varphi(y)+\omega(y,\nabla^{\h}_{x}z)\xi \\
    & \quad -\nabla^{\h}_{y}\nabla^{\h}_{x}z -2\omega(x,y)\varphi(z)+\omega(\theta(x,y),z)\xi-\nabla^{\h}_{\theta(x,y)}z\\
    &=R^{\h}_{x,y}z+\omega(y, z)\varphi(x)-\omega(x,z)\varphi(y)-2\omega(x,y)\varphi(z)-\omega(x,\nabla^{\h}_{y}z)\xi\\
    & \quad +\omega(y,\nabla^{\h}_{x}z)\xi+\omega(\theta(x,y),z)\xi.
\end{split}\end{equation*}

But $\h$ is K{\"a}hler and $\nabla^{\h}$ is skew and torsion-free, so:
\begin{equation*}\begin{split}
    -\omega(x, \nabla_{y}^{\h}z)+\omega(y,\nabla_{x}^{\h}z) & =
    -\langle x,\nabla _{y}^{\h}Jz\rangle+\langle y , \nabla^{\h}_{x}Jz\rangle = 
    \langle \nabla^{\h}_{y}x,Jz\rangle - \langle \nabla^{\h}_{x}y,Jz\rangle \\
    & = -\langle \theta(x,y),Jz\rangle = -\omega(\theta(x,y),z). 
\end{split}\end{equation*}

Substituting this in the previous equation proves the claim.
\end{proof}

\begin{lemma}\label{Ric g y Ric h}
The Ricci curvatures of $\g$ and $\h$ are related by
\[\ric(x,y)=-2\langle  x, y \rangle + \ric^{\h}(x,y), \qquad x,y \in \h.\]
\end{lemma}
\begin{proof} If we use an orthonormal basis  $\{e_i\}$  on  $\h$, 
$$
    \ric(x,y)=\langle R_{\xi,x}y,\xi \rangle + {\textstyle\sum}_i\langle R_{e_i,x}y,e_i\rangle.
$$
Invoking \cite[Lemma 2.2 (iii)]{to} and computations in \cite [Section 2]{AFV},    the first summand equals 
$$
   \langle R_{\xi,x}y , \xi \rangle = -\langle y, R_{\xi,x}\xi \rangle
    =-\langle y , \eta(x)\xi-\eta(\xi)x \rangle  
    = \langle x ,y \rangle. 
$$
The other term is found by routine calculations and Lemma \ref{Curvatura de g y h}:
\begin{equation*}\begin{split}
    {\textstyle\sum}_i\langle R_{e_i,x}y,e_i\rangle&=
    {\textstyle\sum}_i \langle R^{\h}_{e_i,x}y+\omega(x,y)\varphi(e_i)-\omega(e_i,y)\varphi(x)-2\omega(e_i,x)\varphi(y),e_i\rangle \\
    &=\ric^{\h}(x,y)+ {\textstyle\sum}_i \langle \omega(x,y)\varphi(e_i)-\omega(e_i,y)\varphi(x)-2\omega(e_i,x)\varphi(y),e_i\rangle \\
    &=\ric^{\h}(x,y)- {\textstyle\sum}_i \big(\omega(e_i,y)\langle Jx,e_i\rangle+2\omega(e_i,x)\langle Jy,e_i \rangle \big) \\
    &=\ric^{\h}(x,y)-\langle Jy,Jx \rangle - 2 \langle Jx,Jy\rangle \\
    &=\ric^{\h}(x,y)-3\langle x,y \rangle.
\end{split}\end{equation*}
Substituting proves the claim. 
\end{proof}

The next fact explains the perhaps mysterious constant $-2$ that  discriminates the $\eta$-Einstein types of Definition \ref{def:eta-ES}.

\begin{prop}\label{prop: einstein con centro}
Let $\g$ be a Sasakian Lie algebra with centre, and $\h$ the K\"ahler quotient. Then $\g$ is $\eta$-Einstein (with constant $\lambda$) 
if and only if $\h$ is K{\"a}hler-Einstein (with Einstein constant $\lambda +2$).
\end{prop}

\begin{proof}
The $\eta$-Einstein equation reads 
\[\ric(x,y)=\lambda\langle x, y \rangle, \qquad x,y \in \h,\]
and by Lemma \ref{Ric g y Ric h} this happens precisely when 
\[ \ric^{\h}(x,y)=(\lambda+2)\langle x , y \rangle, \]
which in turn is saying $\h$ is K{\"a}hler-Einstein.
\end{proof}

\begin{coro}\label{coro: eta-einstein flat}
A unimodular Sasakian algebra with centre is null $\eta$-Einstein. 
\end{coro}

\begin{proof}
Of course $\g$ is unimodular if and only if $\h$ is. But a unimodular K{\"a}hler algebra must be flat \cite{Hano}, whence 
Ricci-flat a fortiori. Proposition \ref{Ric g y Ric h} forces  $\ric(x,y)=-2\langle x , y \rangle$, making $\g$ null $\eta$-Einstein.
\end{proof}

\begin{example}[null $\eta$-Einstein solvmanifold] \label{ex: null solvmanifold}
{\rm Let $\h=\R e_1 \ltimes_{\ad}\R^5$ be the $6$-dimensional, flat K{\"a}hler Lie algebra spanned by an orthonormal basis $(e_1,e_2)\cup (e_3,e_4)\cup (e_5,e_6)$ and where the derivation
\begin{equation}\label{eq: matrix ad_e1}
    \ad_{e_1}^\h:=
\left(\begin{tabular}{c | c c | c c}
        $0$&$0$&$0$&$0$&$0$\\ \hline
        $0$&$0$&$-a$&$0$&$0$\\
        $0$&$a$&$0$&$0$&$0$\\ \hline
        $0$&$0$&$0$&$0$&$-b$\\
        $0$&$0$&$0$&$b$&$0$
\end{tabular}\right),
    \qquad a,b>0,
\end{equation}
acts on $e_1^\perp=\R^5$. Thus, we are prescribing the essential brackets on $\h$ as
$$[e_1,e_3]=ae_4,\qquad [e_1,e_4]=-ae_3, ,\qquad 
[e_1,e_5]=be_6,\qquad [e_1,e_6]=-be_5.$$
Set $Je_1=e_2$, $Je_3=e_4$,$Je_5=e_6$, so that the fundamental $2$-form is $\omega=-e^{12}-e^{34}-e^{56}$.

Let us build out of the above a Sasakian algebra $\R \xi \oplus \h$ of dimension $7$ with non-trivial centre. We extend the bracket on $\h$ using \eqref{corchete de extension}: this adds to the brackets above the further relations
$$[e_1,e_2]=2\xi,\qquad [e_3,e_4]=2\xi,\qquad [e_5,e_6]=2\xi.$$
In this way we may write 
$$\g=\R e_1 \ltimes_{\ad} \underbrace{(\h_{5}\times \R e_2)}_{\n}$$ 
where $\h_5$ stands for the Heisenberg Lie algebra spanned by $(e_3,e_4,e_5,e_6,\xi)$ (making $\g$ almost nilpotent). 
On $\g$ we now let
\begin{equation*}
    \ad_{e_1}^\g=
    \left(\begin{tabular}{c c c | c }
         & & & $0$\\ 
         & $\ad_{e_1}^\h$ & & $\vdots$\\
         & & & $0$\\  \hline
        $2$ & $0$ & $\cdots$ & $0$
\end{tabular}\right)
\end{equation*}
act on $\n$, and exponentiate it to 
\begin{equation*}
    \op{exp}(t\ad_{e_1}^\g)=
    \begin{spmatrix}
        1&0&0&0&0&0\\
        0&\cos(ta)&-\sin(ta)&0&0&0\\
        0&\sin(ta)&\cos(ta)&0&0&0\\
        0&0&0&\cos(tb)&-\sin(tb)&0\\
        0&0&0&\sin(tb)&\cos(tb)&0\\
        2t&0&0&0&0&1
    \end{spmatrix}
\end{equation*}
The basis $(e_2,e_3,e_4,e_5,e_6,\xi)$ of $\n$ is rational, so by picking 
$t\in\Z$ and $a, b\in\frac{\pi}{2}\Z$, the above matrix becomes integer-valued.
By \cite[Lemma 2.9]{Anna} we conclude that the 1-connected almost nilpotent Lie group $G$ integrating $\g$ admits a lattice $\Gamma$, and thus the compact quotient $\Gamma \backslash G$ is a Sasakian solvmanifold. The latter is null $\eta$-Einstein as $G$ is unimodular.} \par
\quad {\rm The example can be easily generalised to any odd dimension greater than 7 by augmenting \eqref{eq: matrix ad_e1} using skew-symmetric blocks $\begin{pmatrix} 0 & -x \\ x & 0 \end{pmatrix}$, $x>0$, along the diagonal.}
\end{example}

\bigbreak

\section[Centreless Sasakian algebras with $\dim \Im \ad_{\xi}=2$]{Centreless Sasakian algebras  with \texorpdfstring{$\dim \Im  \ad_{\xi}=2$}{}}\label{sec:centreless}

In this section the working assumption is that $\g$ is a $(2n+1)$-dimensional Sasakian Lie algebra with no centre. The Reeb vector no longer commutes with everything, and it was shown in \cite[Corollary 3.12]{AFV} that there is an orthogonal splitting
  \[\mathfrak{g}= \Ker\ad_{\xi} \oplus  \Im \ad_{\xi}.\]
  Let us call these summands 
  $$\s:= \Ker \ad_{\xi}\quad \text{ and }\quad \q:=\Im  \ad_{\xi}.$$  
The kernel of $\ad_{\xi}$ is a Sasakian Lie algebra in which $\xi$ is central \cite[Proposition 3.13]{AFV}. Hence we may apply the reduction process of the previous section to obtain
$$\s= \R \xi \oplus \h_{1},$$ 
where  $(\h_{1}, \theta)$  is K{\"a}hler and 
\begin{equation}\label{eq: h1-bracket}
 [x,y]=-2\omega(x,y)\xi + \theta(x,y), \quad [x,\xi]=0, \qquad x,y \in \h_1.
\end{equation} 
Therefore $\g$ decomposes orthogonally as vector space: 
\begin{equation}\label{eq:g=xi+h1+w}
\mathfrak{g}= \mathbb{R} \xi \oplus \underbrace{\h_{1} \oplus \q}_{\h}.
\end{equation}
From now on, whenever $\h_1$ is considered as a Lie algebra we should keep in mind that the Lie bracket on it is $\theta$, as explained earlier.

Let us begin to examine the finer geometric structure on the above splitting. 
The Jacobi identity for $\xi$, $x \in \h_1$  reads 
\begin{equation}\label{eq: ad-x}
\ad_{x}\restr_{{\q}} \circ \ad_{\xi}\restr_{\q} = \ad_{\xi}\restr_{\q} \circ \ad_x\restr_{\q}.
\end{equation}

\medskip

We shall concentrate on the case 
$$\dim \q=2.$$
We saw at the beginning that the mapping $\ad_{\xi}$ is skew-symmetric and commutes with $\varphi$, so we can find an orthonormal basis $(u,v)$ of ${\q}$ with  $\varphi u=v$ such that 
\begin{equation}\label{eq:ad_xi w}
\ad_{\xi}\restr_{\q}=
\begin{pmatrix} 0 & -k \\ k & 0 \end{pmatrix}
\end{equation}
for some non-zero constant $k$. Then \eqref{eq: ad-x} says there exist linear forms $\mu, \gamma \in \h_{1}^*$ such that 
\begin{equation}\label{eq:ad_x w}
\ad_{x}\restr_{\q}=
\begin{pmatrix} \mu(x) & -\gamma(x) \\ \gamma(x) & \mu(x)
\end{pmatrix},\qquad x\in \h_1.
\end{equation}

To begin with, we claim that 
\begin{equation} \label{[u,v]} 
    [u,v]=2\xi + H_{0}\quad \text{ for some }\ H_{0} \in \h_1.
\end{equation}
To see this, the Jacobi identity implies $[\xi,[u,v]]=0$, so $[u,v] \in \s$ and $[u,v]=t\xi+ H_{0}$ for $t \in \R$ and some vector $H_{0} \in \h_1$. 
The Reeb vector's factor must equal 
$t=t\eta(\xi)= \eta(t\xi + H_{0})= -\D \eta(u,v)= -2\langle u , \varphi (v) \rangle= 2$.

Secondly, applying the Jacobi identity with $u,v$ and $x \in  \h_1$ we obtain
\begin{equation*}\begin{split}
    0 & = \big[x,[u,v]\big]+\big[u,[v, x]\big]+\big[v,[x,u]\big]\\
      &=[x,2\xi +H_{0}]+[u,\gamma(x)u-\mu(x)v]+[v,\mu(x)u+\gamma(x)v]\\
      & = [x,H_{0}]- 2\mu(x) [u,v], 
\end{split}\end{equation*}
so that $ 2\mu(x)(2\xi+H_0)=-2\omega(x,H_0)\xi+ \theta(x,H_0)$ and finally 
\[ \mu(x)=-\frac{1}{2} \omega(x,H_0)=-\frac{1}{2}\langle x , JH_0\rangle,  \quad \text{ and }\quad 2\mu(x)H_0= \theta(x,H_0).\]
Finally, again Jacobi on $x,y \in \h_1$ and $u$ tells 
\begin{equation*}\begin{split}
    0 &= \underaccent{x, y, u}{\mathfrak{S}}\ \big[x,[y,u]\big] \\
      &=[x,\mu(y)u+ \gamma(y)v]+[y, -\mu(x)u-\gamma(x)v]+[u,-2\omega(x,y)\xi+\theta(x,y)]\\
      &=\bigg(\big(2\omega(x,y)k-\gamma\big(\theta(x,y)\big)\bigg) v-\mu\big(\theta(x,y)\big)u
\end{split}\end{equation*}
As $u,v$ were chosen to be independent, either coefficient vanishes.
Put more plainly, 
\[ \omega = -\frac{1}{2k}\,\D^{\h_1} \gamma\]  
is exact and $\mu$ is closed: 
\[\D^{\h_1}\mu =0.\]

\medskip

The normality condition does not provide additional information nor constraints. 
To sum up, we have proved the following. 

\begin{prop}\label{prop: compatibilidad}
Let  $\g$ be a Sasakian Lie algebra with no centre and $\dim \g\geq 5$, written as \eqref{eq:g=xi+h1+w}. If $\dim \Im \ad_{\xi}=2$ then 
    \begin{enumerate}
    \item $\h_1$ is a K{\"a}hler-exact algebra, meaning $\omega=- \dfrac{1}{2k}\,\D^{\h_1}\gamma$;
    \item $\mu\in\h_1^*$ is closed:\ $\mu(x)=-\dfrac{1}{2}\langle x,JH_0\rangle$, and further\  $2\mu(x)H_0= \theta(x,H_0)$.\hfill 
\end{enumerate}
\end{prop}

From $(2)$ it follows that  $0=\mu(\theta(x,y))=-\frac{1}{2}\langle \theta(x,y), JH_0\rangle$, so $JH_0$ is perpendicular to the derived algebra $\theta(\h_1,\h_1)$.
Taking $x=J H_0$ 
gives $-\| H_0 \| ^{2}H_{0}=\theta(JH_{0},H_{0})$, so condition $(1)$ forces
$$    -\| H_{0} \| ^{2}\gamma (H_0)=\gamma\big(\theta(JH_{0},H_{0})\big)
    =-d\gamma(JH_0,H_0)     =2k\omega(JH_0,H_0)     =2k\| H_{0}\| ^{2}.
$$
Therefore, provided $H_0\neq 0$, we end up determining the 
coefficient:
\begin{equation}\label{gamma en H0}
    \gamma(H_0)=-2k
\end{equation}

The above are only necessary conditions, but they explain that we can reverse-engineer the construction, to show the following.

\begin{prop}\label{h1 crea sasaki sin centro}
    Let  $(\h_1, \theta(\cdot,\cdot), \pint)$ be a metric Lie algebra of dimension $2m$  equipped with a K{\"a}hler-exact structure $\omega=-\frac{1}{2k}\D^{\h_1}\gamma$, for some $k\neq 0$, $\gamma \in \h_{1}^*$.
 Assume there exist an element $H_0 \in \theta(\h_1,\h_1)$ obeying the compatibility conditions:
\begin{align}
    & JH_0 \in \theta(\h_1,\h_1)^{\perp}  \label{eq: JH0}\\ 
    & 2\mu(x)H_0=\theta(x,H_0) \quad \text{for every } x \in \h_1, \label{eq: importante}
\end{align}
where $\mu\in\h_1^*$ is the closed 1-form given by $\mu(x)=-\frac12\langle x,JH_0\rangle$, $x\in \h_1$. \par 
Call $\s= \R\xi \oplus \h_1$ the (Sasakian) central extension of $\h_1$ by the cocycle $\omega$, as described in \S \ref{sec:ES with centre}.  
Choose a $2$-plane ${\q}$ with orthonormal basis $(u,v)$ and let 
\[ \g=\s \oplus {\q}\]
with bracket prescribed by \eqref{eq: h1-bracket}, \eqref{eq:ad_xi w}, \eqref{eq:ad_x w} and \eqref{[u,v]}. This makes  $\g$ a Lie algebra of dimension $2m+3$ with $\s$ as subalgebra. \par
Moreover, if we extend the structure $(\varphi, \eta, \xi, \pint)$ on $\s$ to $\g$ by 
$$\varphi (u)=v,\qquad \eta({\q})=0,$$  
and the rest in the natural way, $(\g, \varphi, \eta, \xi , \pint)$ becomes a Sasakian Lie algebra with no centre.
\end{prop}

As a matter of fact, what we have is a full-on characterisation:
\begin{thm}\label{carac sasaki}
Any $(2n+1)$-dimensional Sasakian Lie algebra  $(\g, \varphi, \eta, \xi , \pint)$ with trivial centre and $\dim \Im  \ad_{\xi}=2$ arises by the recipe of the previous proposition.
\end{thm}

\begin{examples}
{\rm Take the $4$-dimensional algebra $ \h_1=\aff(\R) \times \aff(\R)$. Choose an orthonormal basis  $(e_j)$ with brackets
$$\big(0,-e^{12},0,-e^{34}\big)$$
Then $\omega =\D(e^{2}+e^{4}):=-\frac 1{2k}\D\gamma$ defines a K{\"a}hler-exact structure.\par 
\quad i) Pick the zero vector for $H_0$, so that $\mu =0$. The first step  in the extension is to introduce a vector $e_7$ to build $\s=\R e_7 \oplus \h_1 $. Next, we call ${\q}$ the two-plane spanned by vectors $e_5, e_6$, orthonormal by decree. Finally, let $\varphi e_5=e_6$ and  $\varphi\restr_{\h_1}=J$. 
Proposition \ref{h1 crea sasaki sin centro} says that  $\g=\s \oplus {\q}$, with Lie structure
$$\big(0,-e^{12},0,-e^{34}, 
  k (2 e^{2} +2 e ^{4}+e^{7}) e^{6}, k (2 e^{2} +2 e^{4}-e^{7}) e^{5}, 
 -2 (e ^{12}+e^{34}+e^{56}) 
\big),$$
is a $7$-dimensional centreless Sasakian algebra. It cannot  be solvable because $\R e_7 \oplus {\q}$ is simple: $\su(2)$ if $k>0$ and $\mathfrak{sl}(2,\R)$ for $k<0$.\par
\quad ii) Another possible choice is $H_0=e_2$, since $JH_0 \in \theta(\h_1, \h_1)^\perp=\op{span}\{e_1,e_3\}$. This gives $\mu=\frac{1}{2}e^1$, and the end result is now slightly more entangled, algebraically: 
$$\big(0,  -e^{12}-e^{56}, 0, -e^{34}, 
 -k (2 e^{2} +2 e ^{4}-e^{7}) e^{6} -\tfrac{1}{2}e^{15}, 
 k (2 e^{2} +2 e^{4}+e^{7}) e^{5} -\tfrac{1}{2}e^{16}, 
 -2(e^{12}+e^{34}+e^{56}) 
 \big),$$
In contrast to the instance in i), this one is solvable. }
\end{examples}
These two examples are indicative of a pattern that we shall address next. 

\bigbreak

\subsection{Issues of solvability} 
\medbreak

We wish to discuss the solvability of a centreless Sasakian Lie algebra $\g$ of the type characterised in Theorem \ref{carac sasaki}. 
 It turns out that the vector $H_0$ appearing in \eqref{[u,v]} plays an important role in the matter.

\begin{prop}\label{cond g soluble}
A Sasakian Lie algebra $\g$ with no centre and $\dim \Im \ad_{\xi}=2$ is solvable, if and only if $H_{0}\neq 0$ and $\h_1$ is solvable.
\end{prop}
\begin{proof} 
$(\Longrightarrow$)\ As $\g$ is solvable, the derived algebra $[\g,\g]$ is nilpotent. 
We argue by contradiction and suppose $H_{0}= 0$. Then  equation 
\eqref{[u,v]} says $\xi \in [\g,\g]$ and therefore $\ad_{\xi}$ is nilpotent. It is also skew as the Reeb vector is Killing, so it should vanish altogether, against the assumption. This proves $H_0$ cannot vanish. 
But $\s$ is a subalgebra in $\g$ so it is solvable. It also is a central extension of $\h_1\iso \bigslant{\s}{\R\xi}$, so $\h_1$ must be solvable too.\\
$(\Longleftarrow$)\ Now, recall that $\s=\Ker \ad_\xi$ and we fixed a basis $(u,v)$ for $\q$.   
Let us inspect the derived series, beginning from 
\begin{equation*}\begin{split}
\g^{(1)} &= \text{span}\{2\xi+H_{0}, u, v\} + \s^{(1)},\\
\s^{(1)} &= \{-2\omega(x,y)\xi+\theta(x,y) \mid  x,y \in \h_1\}.
\end{split}\end{equation*}
We have $[2\xi+H_0,{\q}]=0$ 
because 
$$[2\xi+H_{0},u] = 2[\xi,u]+[H_{0},u] = 
2(kv)+\mu(H_0)u+\gamma(H_0)v \stackrel{\eqref{gamma en H0}}{=} 
2kv-2kv  =0,$$
and similarly $[2\xi+H_0,v]=0$.  Secondly, $[2\xi+H_0,\s^{(1)}]=0$, for  
\begin{equation*}\begin{split}
  [2\xi+H_{0}, -2\omega(x,y)\xi+\theta(x,y)] &=[H_{0},\theta(x,y)]
    =-2\omega\big(H_{0},\theta(x,y)\big)\xi+\theta\big(H_{0},\theta(x,y)\big)\\
    &=2\langle JH_0, \theta(x,y)\rangle\xi-\theta\big(\theta(x,y),H_0\big) \\
 &   \stackrel{\text{Prop.} \ref{prop: compatibilidad}}{=} -2\mu\big(\theta(x,y)\big)H_0  =0.
\end{split}\end{equation*}
Finally, $[{\q},\s^{(1)}]=0$ since
\begin{equation*}\begin{split}
    [u,-2\omega(x,y)\xi+\theta(x,y)]& 
    =2\omega(x,y)(kv)-\gamma\big(\theta(x,y)\big)v \\
  &  =2\omega(x,y)(kv) + d^{\h_1} \gamma(x,y)v \\
&     \stackrel{\text{Prop.} \ref{prop: compatibilidad}}{=} 2\omega(x,y)(kv)-\big(2k\omega(x,y)\big)v  =0,
\end{split}\end{equation*}
and $[v,\s^{(1)}]=0$ alike.\\
From all of that we deduce $\g^{(2)}=[{\q},{\q}]+\s^{(2)}=\text{span}\{2\xi+H_0\}+\s^{(2)}$ 
and hence $\g^{(3)}=\s^{(3)}$. In other words the derived series of  $\g$ and $\s$ agree from the third step onwards. But $\s$ is solvable since $\h_1\iso \bigslant{\s}{\R\xi}$ is, by hypothesis.
\end{proof}
\bigbreak

If we choose $H_{0}=0$ in the construction, $\g$ will therefore have a semisimple part in its Levi decomposition. Indeed,  

\begin{lemma}\label{lemma: radical}
Let $\g$ be a Sasakian Lie algebra with trivial centre and $\dim \Im \ad_{\xi}=2$. If $H_0=0$ then 
$\rad(\g)= \left\{ -\dfrac{\gamma(h)}{k} \xi+h \mid  h\in \rad(\h_1) \right\}.$
\end{lemma}

\begin{proof}
Let us recall that 
\begin{equation*}\label{radical formula}
       \rad(\g)=[\g,\g]^{\perp _{\beta}}=\{z \in \g \mid \beta(z, y)=0 \, \text{ for all}\; y \in [\g,\g]\},
\end{equation*}
where $\beta$ is the Cartan-Killing form of $\g$.  
 Having $H_0=0$ implies $\mu=0$ and $[u,v]=2\xi$, so by Proposition \ref{h1 crea sasaki sin centro} 
\begin{equation}\label{conmutador de no soluble}
    [\g,\g]=\R\xi \oplus [\h_1, \h_1] \oplus {\q}. 
\end{equation}
Let us, for clarity, write the matrices used to compute $\beta$ in the orthonormal frame $(\xi, e_1,\ldots, e_{2n},u,v)$ of $\g$, where $(e_i)$ is an orthonormal basis of $\h_1$. 
The matrix of $\ad_{\xi}$ is all zero except for the diagonal ${\q}$-block where 
\eqref{eq:ad_xi w} holds.
For any $h\in\h_1$, 
\[
\ad_{h}=
\left(\begin{array}{c|ccc|cc}
          0  &\cdots & 2\langle Jh,e_i\rangle & \cdots &0&0 \\
          \hline 
          0&&& &0&0\\
          \vdots&&\ad^{\h_1}_h&&\vdots&\vdots\\
          0&&&&0&0\\
          \hline
       0 & \cdots & 0 & \cdots &0 & -\gamma(h) \\
        0 & \cdots & 0 & \cdots &\gamma(h) & 0
        \end{array}\right).
\]
Moreover,
\[
\ad_{u}=
\left(\begin{array}{c|ccc|cc}
          0  &\cdots & 0 & \cdots&0&2 \\
          \hline 
          0&&& &0&0\\
          \vdots&&0&&\vdots&\vdots\\
          0&&&&0&0\\
          \hline
       0 & \cdots & 0 & \cdots&0 & 0 \\
        -k &  \cdots & -\gamma(e_i) &  \cdots
        &0 & 0
        \end{array}\right), \quad 
\ad_{v}=
\left(\begin{array}{c|ccc|cc}
          0  &\cdots & 0 & \cdots&-2&0 \\
          \hline 
          0&&& &0&0\\
          \vdots&&0&&\vdots&\vdots\\
          0&&&&0&0\\
          \hline
       k &  \cdots & \gamma(e_i) &  \cdots
       &0 & 0 \\
        0 & \cdots & 0 & \cdots&0 & 0
        \end{array}\right).
\]
This gives the values of $\beta$  (for $h,\tilde{h}\in\h_1$):
$$\begin{array}{ll}
\beta(\xi,\xi)=-2k^2, & \; \beta(u,u)=-4k, \qquad \beta(v, v)=-4k\\
\beta(\xi, h)=-2k\gamma(h), & \;
\beta(h,\widetilde{h})=\beta^{\h_1}(h,\widetilde{h})-2\gamma(h)\gamma(\widetilde{h}).
\end{array}$$

Pick an element $z=a\xi+ bu +cv+h \in \rad(\g)$, where $a,b,c\in\R$, $h\in\h_1$. Using decomposition \eqref{conmutador de no soluble} we obtain
\begin{equation*}\begin{split}
    0=\beta(z,\xi)&=-2ak^{2}-2k\gamma(h) \then \gamma(h)=-ak \\
    0=\beta(z,u)&=-4b k \then b=0\\
    0=\beta(z, v)&=-4ck \then c=0.
\end{split}\end{equation*}
Consequently for $h' \in [\h_1, \h_1]$
$$0=\beta(z, h')=-2ak\gamma(h')+\beta^{\h_1}(h, h') -2\gamma(h)\gamma(h') = \beta^{\h_1}(h, h'),$$
which says $h \in \rad(\h_1)$.
\end{proof}

\begin{thm}
Let $\g$ be a Sasakian Lie algebra with trivial centre and $\dim \Im \ad_{\xi}=2$ such that $\h_1$ is solvable and $H_{0}=0$ (thus $\g$ is not solvable). Then the Levi decomposition 
$$\g\cong (\R\xi\oplus {\q}) \times \rad(\g)$$ 
is a direct product. Moreover, $\rad(\g)$ is 
isomorphic to $\h_1$ and the semisimple part is  isomorphic to 
either to $\su(2)$ or $\mathfrak{sl}(2,\R)$.
\end{thm}
\begin{proof}
Since $\h_1$ is solvable it equals its radical $\rad(\h_1)$, and it follows from Lemma \ref{lemma: radical} that 
\[\rad(\g)= \left\{-\frac{\gamma(h)}{k} \xi+ h\mid h\in\h_1\right\}. \]
Let us denote $\mathfrak b:=\text{span}\{\xi, u, v\}=\R\xi\oplus {\q}$. 
By virtue of \eqref{eq:ad_xi w}, it clearly is a subalgebra 
isomorphic to $\su(2)$ or $\mathfrak{sl}(2,\R)$, depending on the sign of $k$. \par
\quad In order to show that the product is direct rather than semi-direct, we will prove that $[\mathfrak{b}, \rad(\g)]$ is zero. Consider a radical element $a\xi+h$ against the generic element in $\mathfrak{b}$. 
First, $\ad_\xi(\h_1)=0$ implies $[\xi, \rad(\g)]=0$. 
Similarly, $[{\q}, \rad(\g)]=0$ by bracketing with $u,v$. \par
\quad To finish the proof pick a unit vector $e_1\in \g$ with dual $ e^1= \frac 1c \gamma$ for some real number $c>0$. Extending $e_1$ to an orthonormal basis $\{e_1, e_2, \ldots,  e_{2n}\}$ of $\h_1$ we find
\[\rad(\g)=\text{span}\left\{-\frac{c}{k}\xi+e_1, e_2, \ldots,  e_{2n}\right \},\]
which is patently isomorphic to $(\h_1,\theta)$ since $\h_1\subseteq \Ker\ad_\xi$.
\end{proof}

\begin{prop}\label{g unimodular}
     $\mathfrak{g}$ is unimodular  if and only if
$\tr \ad ^{\h_1} _{x}=-2\mu(x)$, for every  $x \in \h_1$.
\end{prop}

\begin{proof}
If  $x\in \h_1$ it follows that 
\[
\ad_{x}=
\left(\begin{array}{c|c}
          \ad^{\s}_x  & \\          \hline 
        & \ad_{x}\restr_{\q}
        \end{array}\right)
\]
and 
\begin{equation*}\begin{split}
    \tr \ad^{\s} _{x}= \langle [x,\xi],\xi \rangle+ {\textstyle\sum}_i \langle [x, e_i],e_i \rangle
   & ={\textstyle\sum}_i \langle \theta(x,e_i) -2 \omega(x,e_i)\xi , e_i \rangle \\
   & ={\textstyle\sum}_i \langle \theta(x,e_i),e_i \rangle
    =\tr\ad^{\h_1} _{x}.
\end{split}\end{equation*}
Hence  $\tr \ad_{x}=\tr \ad^{\h_1} _{x}+2\mu(x)$ for all $x \in \h_1$, from \eqref{eq:ad_xi w}. On the other hand, the full adjoint map of  $x \in {\q}$ only has off-diagonal blocks, whence it is trace-free.
Moreover, $\ad_{\xi}$ is traceless because of \eqref{eq:ad_xi w} and $[\xi , \s]=0$. All that reduces the unimodularity to the simpler constraint 
$\tr \ad^{\h_1} _{x}=-2\mu(x)$, for $x \in \h_1$.
\end{proof}

\smallskip

Under the previous assumptions, 
\begin{coro}
$\g$ cannot be unimodular in case $H_0=0$.  
\end{coro}
\begin{proof} 
Choosing $H_0$ to be zero means $\mu(x)=0$, so if $\g$ were unimodular the previous proposition would make $\ad^{\h_1}$ traceless on $\h_1$. That, though,  is not possible since K{\"a}hler-exact algebras are not unimodular by a result of \cite{frob}. 
\end{proof}

\bigbreak 

\subsection{Ricci tensors}\label{sec: ricci} 
\medbreak

We begin by computing the Ricci tensor of $\g$ by brute force. 
Proposition \ref{nabla xi} furnishes the following summary:
\begin{longtable}[c]{c | c c c c} 
{} & $\xi$ & $y$ & $u$ & $v$\\ \midrule
$\nabla_{\xi}$ & $0$ & $-Jy$ & $(k-1)v$ & $(1-k)u$\\[1mm] 
$\nabla_{x}$ & $-Jx$ & $-\omega(x,y)\xi +\nabla_{x}^{\h_1}y$ & $(\mu(Jx)+\gamma(x))v$ & $-(\gamma(x)+ \mu(Jx) )u$ \\[1mm] 
$\nabla_{v}$ & $u$ & $-\mu(Jy)u-\mu(y)v$ & $-\xi-\tfrac{1}{2}H_0$ & 
$-\tfrac{1}{2}JH_0$\\[1mm] 
$\nabla_{u}$ & $-v$ & $-\mu(y)u+\mu(Jy)v$ & $-\tfrac{1}{2}JH_0$ & 
$\xi + \tfrac{1}{2}H_0$\\ \midrule
\end{longtable}
with $x,y \in \h_1$. 
Then 
\begin{equation*}
    \ric (x,y)=\ric^{\h_1}(x,y)-2\langle x,y\rangle+
    \langle \nabla_{y}^{\h_1}x, JH_0 \rangle +\frac{1}{2}\langle H_0,x\rangle \langle H_0,y\rangle  -\frac{1}{2} \langle JH_0,x \rangle \langle JH_0 , y \rangle.
\end{equation*}
Observe that this is indeed symmetric as $JH_0$ is orthogonal to $\h_1^{(1)}$, so $ \langle JH_0, \nabla^{\h_1}_{y}x\rangle = \langle JH_0 , \nabla ^{\h_1}_{x}y\rangle$. Moreover, using the Koszul formula for the Levi-Civita connection and Proposition \ref{prop: compatibilidad}(2) we may simplify the long formula above to the more significant expression
\begin{equation} \label{Ric en h1}
    \ric (x,y)=\ric^{\h_1}(x,y)-2\langle x,y\rangle+ \frac{1}{2}\langle \theta (Jx, y),H_0 \rangle. 
\end{equation} 
Next, on ${\q}$:
\[ \ric (u,v)= \frac{1}{2} {\textstyle\sum}_i \big(\langle \nabla ^{\h_1}_{e_i}H_0, e_i \rangle + \langle H_0 , JH_0 \rangle\big)= 
\frac{1}{2} {\textstyle\sum}_i \langle \nabla ^{\h_1}_{e_i}H_0, e_i \rangle, \]
and \eqref{eq: importante} tells 
$$
   {\textstyle\sum}_i \langle \nabla ^{\h_1}_{e_i}H_0,e_i \rangle ={\textstyle\sum}_i \langle \theta(e_i,H_0), e_i \rangle 
   =-{\textstyle\sum}_i\langle e_i , JH_0 \rangle \langle H_0 , e_i \rangle 
   =-\langle JH_0,H_0\rangle =0.
$$
Similarly, we have 
\begin{equation}\label{Ric en u,u} 
    \ric (u, u) = -\| H_0 \|^{2} +\gamma(H_0) +2(k-1) - \frac{1}{2} {\textstyle\sum}_i \langle\nabla ^{\h_1}_{e_i} JH_0 , e_i \rangle.
\end{equation} 
The right-hand side provides an expression for $\ric (v, v)$, too. 
The other relevant Ricci terms are easy: $ \ric (x, u)=\ric (x, v)=0$. 
\smallbreak
With these preparations in place, we can proceed to refine Proposition \ref{prop: einstein con centro}.

\begin{prop}\label{H=0}
If $H_0=0$ then 
$$\g \text{ is $\eta$-Einstein } \iff \h_1 \text{ is K{\"a}hler-Einstein}.$$
\end{prop}

\begin{proof}
Let us enforce the above computations under the new assumption, to find
\begin{equation*}\begin{array}{c}
\ric (x,y)=\ric^{\h_1}(x,y)- 2 \langle x, y \rangle, \quad 
\ric (u,u)=\ric (v,v)=2(k-1), \\[2mm]
\ric (u,v)=\ric (x,u)=\ric (x,v)=0,
\end{array}\end{equation*}
which we write more effectively as matrix on $\h=\h_1\oplus{\q}$:
\begin{equation*}
\ric =
\left( \begin{array}{c|c}
     \ric^{\h_1} - 2\Id  & \\ \hline
     & 2(k-1)\Id 
\end{array} \right).
\end{equation*}
Immediately, $\g$ is $\eta$-Einstein precisely when $\ric^{\h_1}(x,y)=2k\langle x,y\rangle$. 
\end{proof}

\bigbreak 

There remains to see to the generic, and more difficult, case in which $H_0 \neq 0$. Before going head-on into that matter, let us observe the following.

\begin{coro}
If $\g$ is $\eta$-Einstein with $H_{0}\neq 0$, the Einstein constant equals
\begin{equation} \label{lambda para eistein}
    \lambda=\frac{1}{2}\tr\ad^{\h_1}_{JH_0} -\| H_0 \|^{2}-2.
\end{equation}
 \end{coro}
\begin{proof}
We pick up from \eqref{Ric en u,u} and push the computation further:
\begin{equation*}\begin{split}
    \langle \nabla^{\h_1}_{e_i}JH_0, e_i \rangle &=\langle J \nabla ^{\h_1}_{e_i}H_0, e_i \rangle =- \langle \nabla^{\h_1}_{e_i}H_0, Je_i \rangle\\
    &=-\frac{1}{2} \big( \langle \theta (e_i,H_0) , Je_i \rangle + \theta (Je_i, H_0) ,e_i \rangle + \langle \theta(Je_i, e_i) , H \rangle \big)\\
    &=-\frac{1}{2} \big( -\langle e_i , JH_0 \rangle \langle H_0 , Je_i \rangle - \langle Je_i, JH_0 \rangle \langle H_0 , e_i \rangle + \langle \theta(Je_i, e_i) , H_0 \rangle\big) \\
    &=-\frac{1}{2} \big( \langle JH_0, e_i \rangle ^{2}- \langle H_0, e_i \rangle ^{2} + \langle \theta(Je_i, e_i) , H_0 \rangle \big),
\end{split}\end{equation*}
and adding up gives us 
\begin{equation*}\begin{split}
    {\textstyle\sum}_i \langle \nabla^{\h_1}_{e_i}JH_0, e_i \rangle &=
    -\frac{1}{2}\big(\|JH_0\|^2-\|H_0\|^2 + {\textstyle\sum}_i\langle \theta(Je_i, e_i) , H_0 \rangle \big)\\
      &=-\frac{1}{2} {\textstyle\sum}_i  \langle \theta(Je_i, e_i) , H_0 \rangle  
= (\D^{*} \omega)(H_0) - \tr \ad^{\h_1}_{JH_0}
\end{split}\end{equation*}
in terms of the codifferential on $\h_1$ (see for instance \cite[Section 2]{ab}).
Of course $\D^\ast \omega=0$ since $\h_1$ is K{\"a}hler, and so:
\begin{equation} \label{ad y traza}
    {\textstyle\sum}_i  \langle \nabla^{\h_1}_{e_i}JH_0, e_i \rangle=- \tr \ad^{\h_1} _{JH_0}.
\end{equation}

Substituting what we have just found into \eqref{Ric en u,u} and using \eqref{gamma en H0} yields
\begin{equation}
\ric (u,u)=\ric (v, v)=-\| H_0 \|^{2}-2+\frac{1}{2}\tr\ad^{\h_1}_{JH_0}
\end{equation}
and proves the statement.
\end{proof}

The punchline at the end of this section is that in either case ($H_{0}=0$ and $H_{0} \neq 0$) we need to study K{\"a}hler-Einstein Lie algebras with exact K{\"a}hler form. We shall do just that in the subsequent sections, under the additional simplifying assumption that the Sasakian Lie algebra under consideration  (hence the K{\"a}hler one as well) is solvable. This leads to studying \textit{normal $j$-algebras}, which we shall pursue next.

\bigbreak 


\section{Normal \textnormal{j}-algebras}\label{sec:normal-j-algebras}

\begin{definition} \label{def= normal j-algebra}
Let  $\d$ be a completely solvable real Lie algebra, $J\in\End \d$ an integrable complex structure and $f\in\d^*$ a $1$-form. 
The triple $(\d, J, f)$ is called a {\bf normal $\boldsymbol j$-algebra} if 
$$f [Jx,Jy] =f [x,y] \quad \text{ and }\quad f [Jx, x] >0$$ for all  $x, y\in\d$.
\end{definition}
   
Such a gadget admits a $J$-invariant, positive-definite inner product
\[ \langle x ,y \rangle := f [Jx, y]. \]
The fundamental form associated with the Hermitian pair  $\big(J, \pint\big)$   
is therefore $\omega = -\D f$. 
In other words, a normal $j$-algebra is a  completely solvable, K{\"a}hler-exact  Lie algebra. The following structure theorem is part of Piatetskii-Shapiro's seminal work.
\begin{thm}[\cite{shapiro}]\label{prop raices}
Let  $(\d, J, f)$ be a normal $j$-algebra, and write 
$$\d =\n \oplus \a$$ 
where $\n=[\d, \d]$ and $\a$ is the orthogonal complement. Then
\begin{itemize}
    \item $\a$ is an Abelian subalgebra and $\n=\displaystyle\bigoplus_{\alpha\in \Delta} \n_\alpha$ splits into root spaces
    \[ \n_{\alpha}=\{x \in \n \mid [h, x]=\alpha(h)x  \text{ for all } h \in \a \}\] 
with $[\n_{\alpha}, \n_{\beta}] \subseteq \n_{\alpha + \beta}$. 
    \item Let $\epsilon_1 , \ldots,  \epsilon_r$ indicate the roots whose root spaces $J$ maps into $\a$. 
    Then $r=\dim \a\leq \dim \n$. Up to relabelling, furthermore, all roots are of the form
$$\tfrac{1}{2} \epsilon_k, \; \epsilon_k, \text{ for } 1 \leq k \leq r,  
\quad \text{or} \quad 
 \tfrac{1}{2} ( \epsilon_{i} \pm \epsilon_{j}), \text{ for } 1 \leq i <j \leq r$$
(but some of these combinations may not be roots). 
\item If $ x \in \n_{(\epsilon_{i} - \epsilon_{j})/2}$ and $ y \in \n_{\epsilon_j/2}$, or 
$ y \in \n_{\epsilon_j}$, are non-zero, then  $[x,y]\neq 0$,
\item $J\n_{\epsilon_k/2}=\n_{\epsilon_k/2}$, \; $J\n_{(\epsilon_{i} + \epsilon_{j})/2}=\n_{(\epsilon_{i} - \epsilon_{j})/2}$.
\end{itemize}
\end{thm}

This big theorem should be accompanied by D'Atri's result that the root spaces are pairwise orthogonal \cite{D}. The roots $\epsilon_1 , \ldots,  \epsilon_r$ are called {\bf distinguished} and are linearly independent. 
Moreover, each $\n_{\epsilon_k}$ has dimension one \cite{D}, so one may pick generators $x_k \in \n_{\epsilon_k} $ such that $\{ Jx_k \; ; \; 1\leq k \leq r \}$ is a basis of $\a$ satisfying 
\begin{equation}\label{anular raices}
    \epsilon_{k}(Jx_i)=0\ (i\neq k), \quad \epsilon_{k}(Jx_k)\neq 0.
\end{equation}
Even better, 
we have $[Jx_k, x_k]=\epsilon_{k}(Jx_k)x_k$, so $\omega (Jx_k, x_k)=f\big([Jx_k, x_k]\big)=\epsilon_{k}(Jx_k)f(x_k)$ and therefore
\begin{equation} \label{base ortonormal}
    0<\| x_k \|^2 
    =\epsilon_{k}(Jx_k)f(x_k).
\end{equation}
Hence  $\epsilon_{k}(Jx_k)$ and $f(x_k)$ have the same sign. And finally, by changing the generators' signs if necessary, we may even assume 
\begin{equation}\label{f(x_k)<0}
   \epsilon_{k}(Jx_k)>0, \quad  f(x_k)>0.
\end{equation}
This observation will become handy later. 
\bigbreak

From now on we will use the following notation, for $1\leq k\leq r$ and $1\leq i<j\leq r$:
\begin{equation}\label{eq: dimensiones} 
n_{k}:=\dim \n_{\epsilon_k/2}\in 2\Z, \quad \n_{ij}:=\dim \n_{(\epsilon_{i}-\epsilon_{j})/2}=\dim \n_{(\epsilon_{i}+\epsilon_{j})/2}. 
\end{equation}
We will also write $\Delta':=\Delta\setminus \{\epsilon_1 , \dots , \epsilon_r\}$ for the set of non-distinguished roots.
\medskip

\begin{example} 
{\rm
Let $\d$ be the $6$-dimensional completely solvable Lie algebra 
$$(0, 0, -e^{13}+e^{23}, -2 e^{14}- 2 e^{36}, -2 e^{25}, -e^{16}-e^{26}-e^{35})
$$
with orthonormal basis $(e_j)$. 
The complex structure $Je_1=e_4, \, Je_2=2e_5,\, Je_3= e_6$ and the 1-form $f=-\frac12 e^4-\frac14 e^5$ turn it into a normal $j$-algebra.   
To find the root spaces note that 
$$\n=\text{span}\{e_3,e_4,e_5,e_6\} \;\text{ and }\;  \a=\text{span}\{e_1,e_2\}.$$
This means there are two distinguished roots $\epsilon_1, \epsilon_2$.  
Observing that 
$$
    [e_1,e_4]=\epsilon_{1}(e_1)e_4=2e_4\qquad 
    [e_2,e_5]=\epsilon_{2}(e_2)e_5=2e_5,
$$
allows to discover that $\epsilon_{1}=2e^{1}$ and $\epsilon_{2}=2e^{2}$.\par
\quad Regarding non-distinguished roots $\alpha$, let us examine what the possibilities are.
If $\alpha=\tfrac{1}{2}\epsilon_1$, the root space is even-dimensional for being $J$-invariant, so it must be $\n_{\alpha}=\text{span}\{e_3,e_6\}$. This would force  $[e_2,e_3]=0$, a contradiction. A similar reasoning shows $\tfrac{1}{2}\epsilon_2\notin\Delta'$. Therefore a non-distinguished root might only be one of $\tfrac{1}{2}(\epsilon_{1} \pm\epsilon_{2})$, and in either case Proposition \ref{prop raices} gives 
$$\n=\n_{(\epsilon_1-\epsilon_2)/2}\oplus \n_{(\epsilon_1+\epsilon_2)/2} =  \R e_3 \oplus \R e_6.$$}
\end{example}

\bigbreak 

\subsection{Derivations of normal j-algebras} 

If $(\d,J,\pint)$ is an arbitrary Hermitian Lie algebra we denote the space of \textit{unitary} derivations by 
\[ \op{Der}_{u}(\d):=\{ D \in \op{Der}(\d)  \mid DJ=JD \text{ and } D \text{ skew-adjoint}\}.\]

\begin{lemma} \label{Da=0}
    Let  $\d=\n\oplus \a$ be a  normal $j$-algebra. Any skew derivation $D$ kills the Abelian part: $D(\a)=\{0\}$ (and hence $D(\d)\subseteq \n$). 
\end{lemma}
\begin{proof}
For $h \in \a$ and $0\neq x \in \n_{\alpha}$ we have  $D[h, x]=[Dh, x]+[h,Dx]$. 

The commutator ideal $\n$ is $D$-invariant, and since $D$ is skew it must preserve the orthogonal complement of $\n$, namely $\a$. Therefore $D h\in\a$. Now, 
\begin{equation*}
   \alpha(h)Dx=\alpha(Dh)x+[h,Dx].
\end{equation*}
We may split $Dx= \displaystyle\sum_{\beta \in \Delta}(Dx)_{\beta}$ along 
$\displaystyle\bigoplus_{\beta\in \Delta} \n_\beta$, so that the above condition refines to 
\[  \alpha(h) \sum_{\beta \in \Delta} (Dx)_{\beta}=\alpha(Dh)x+\sum_{\beta \in \Delta} \beta(h)(Dx)_{\beta}.
\]
If $\beta= \alpha$ the component in $\n_{\alpha}$ equals 
$    \alpha(h)(Dx)_{\alpha}=\alpha(Dh)x+\alpha(h)(Dx)_{\alpha}$,
which implies $\alpha(Dh)x=0$ and so $\alpha(Dh)=0$. 
This holds for any $\alpha \in \Delta$, and $\Delta$ contains a basis of $\a^*$, forcing $Dh=0$.
\end{proof}

\begin{lemma} \label{D cero}
If $\d$ is a normal $j$-algebra, any unitary derivation kills each distinguished root space:
$$\forall\,D\in \op{Der}_u(\d), \quad  D(\n_{\epsilon_i})=\{0\},
\quad i=1,\ldots,  r.$$
\end{lemma}

\begin{proof}
As $J \n_{\epsilon_i} \subset \a$, it follows from Lemma \ref{Da=0} that $D J (\n_{\epsilon_i})=\{0\}$. But $D$ and $J$ commute so $J D \n_{\epsilon_i}=\{0\}$, ie $D$ kills $\n_{\epsilon_i}$.
\end{proof}

\medskip
In the light of that, unitary derivations might only salvage non-distinguished root spaces. In the following two technical results we analyse this action in detail.

\begin{lemma} \label{D invariante}
If $D \in \op{Der}_{u}(\d)$ then   $D(\n_{\epsilon_i/2}) \subseteq \n_{\epsilon_i/2}$.
\end{lemma}

\begin{proof}
Pick $u_i \in \n_{\epsilon_i/2}$. As $D$ maps it into $\bigoplus_{\alpha} \n_\alpha$, as before we write $Du_i=\sum_{\alpha}(Du_i)_{\alpha}$. 
For starters, $\langle Du_i,  \n_{\epsilon_j}\rangle = - \langle u_i, D\n_{\epsilon_j} \rangle=0$ for any $j$, due to Lemma \ref{D cero}. Hence all  $(Du_i)_{\n_{\epsilon_j}}$ vanish, and as a consequence $Du_i=\sum_{\alpha \in \Delta'}(Du_i)_{\alpha}$. 
\par
\quad We claim that $(Du_i)_{\alpha}=0$ for all $\alpha\in \Delta'$, $\alpha \neq \epsilon_i/2$. Pick $h_j =Jx_j \in \a$, where $\n_{\epsilon_j}= \R x_j$ ($i \neq j)$. Then $D[h_j,u_i]=D\left (\tfrac{1}{2}\epsilon_i(h_j)u_i\right)=0$ by \eqref{anular raices}. On the other hand, Lemma \ref{Da=0} tells us 
$D[h_j,u_i]=[Dh_j,u_i]+[h_j,Du_i]=[h_j,Du_i]$, and 
\begin{equation*}
    [h_j,Du_i]=\left[[h_j, \sum_{\alpha \in \Delta'}(Du_i)_{\alpha}\right]=\sum_{\alpha \in \Delta'}[h_j,(Du_i)_{\alpha}].
\end{equation*}
Hence $ \sum_{\alpha \in \Delta'} \alpha(h_j)(Du_i)_{\alpha}=0$, 
which implies
\[ \alpha(h_j)(Du_i)_{\alpha}=0 \quad \text{for all} \quad \alpha \in \Delta'. \]

If $\alpha$ equals $\tfrac{1}{2} \epsilon_j$ or $\pm\tfrac{1}{2}(\epsilon_{j} \pm \epsilon_{k})$ (any sign combination) 
for $k\neq j$, necessarily $(Du_i)_{\alpha}=0$. But that holds for any $j \neq i$, so the only chance for $(Du_i)_{\alpha}$ to be non-zero is that $\alpha = \epsilon_i/2$, meaning $Du_i \in \n_{\epsilon_i/2}$.
\end{proof}

\begin{lemma} \label{D suma}
If $D \in \op{Der}_{u}(\d)$ then   $D(\n_{(\epsilon_{i}\pm \epsilon_{j})/2})\subseteq \n_{(\epsilon_{i}-\epsilon_{j})/2}\oplus \n_{(\epsilon_{i}+\epsilon_{j})/2}$ for any $i<j$.
\end{lemma}
\begin{proof}
We will prove the result only for $\tfrac{1}{2} (\epsilon_i+\epsilon_j)$, as the other case is a straightforward consequence of $[D, J]=0$. 
Let $u_{ij} \in \n_{(\epsilon_{i}+\epsilon_{j})/2}$ with $i<j$, so as customary 
\[ Du_{ij}=\sum_{\alpha \in \Delta'}(Du_{ij})_{\alpha}, \quad \text{with} \quad (Du_{ij})_{\alpha}\in \n_\alpha, \]
because $\langle Du_{ij},\n_{\epsilon_k}\rangle=-\langle u_{ij},D\n_{\epsilon_k}\rangle=0$ by Lemma \ref{D cero}. Pick a generator for $\n_{\epsilon_k}=\R x_k$ (with $k \neq i,j$) and set $h_k=Jx_{k} \in \a$, so that 
\[D[h_k,u_{ij}]=D\left(\frac{1}{2}(\epsilon_{i}+\epsilon_{j})(h_k)u_{ij}\right)=0 \]
by virtue of \eqref{anular raices}. At the same time, Lemma \ref{Da=0} guarantees that 
\[ D[h_k,u_{ij}]= 
[h_k,Du_{ij}]=\left[h_k, \sum_{\alpha \in \Delta'} (Du_{ij})_{\alpha}\right]=\sum_{\alpha \in \Delta'}\alpha(h_k)(Du_{ij})_{\alpha}. \]
It follows that $\alpha(h_k)(Du_{ij})_{\alpha}=0$ for all $\alpha \in \Delta'$.

If $\alpha = \tfrac{1}{2} \epsilon_k$ or $\pm\tfrac{1}{2}(\epsilon_{k}\pm \epsilon_{l})$ with $l\neq k$, then $(Du_{ij})_{ \alpha}=0$. But that holds for all  $k\neq i,j$, so we only need to verify that $(Du_{ij})_{ \alpha}=0$ when $\alpha= \epsilon_i/2$ or $\epsilon_j/2$. 
Now, for any $v_{i} \in \n_{{\epsilon_i/2}}$ Lemma \ref{D invariante} says 
$Dv_i\in \n_{\frac12\epsilon_i}$ and so 
 \[ \langle Du_{ij},v_i \rangle=-\langle u_{ij},Dv_i \rangle =0.\] 
Consequently  $(Du_{ij})_{\epsilon_i/2}=0$. Swapping the indices' roles  produces $(Du_{ij})_{\epsilon_j/2}=0$ as well. All in all
\begin{equation*}
    Du_{ij}=(Du_{ij})_{(\epsilon_{i}-\epsilon_{j})/2} + (Du_{ij})_{(\epsilon_{i}+\epsilon_{j})/2},
\end{equation*}
completing the proof.
\end{proof}
For future reference, we summarise the previous four lemmas as follows: unitary derivations of normal $j$-algebras annihilate the Abelian part $\a$ and all distinguished root spaces $\n_{\epsilon_i}$, while they preserve every $\n_{\epsilon_i/2}$ and the sums $\n_{(\epsilon_{i}-\epsilon_{j})/2}\oplus \n_{(\epsilon_{i}+\epsilon_{j})/2}$ (non-distinguished roots).
\bigskip

As we showed earlier on, normal $j$-algebras are special instances of  `symplectic-exact' Lie algebras. Hence they are not unimodular by Hano's theorem \cite{Hano} (see \cite[Proposition 3.4]{frob} as well). Examples of non-unimodular, completely solvable, invariant K{\"a}hler structures are provided in \cite{Gindikin-PS-V}.
Notwithstanding, we can prove that fact directly from definition \ref{def= normal j-algebra}.

\begin{prop} \label{traza en a}
Let $\d=\a\oplus \bigoplus_{\alpha} \n_{\alpha}$ be a normal $j$-algebra, and suppose $\{h_1, \ldots,  h_r \}$ a basis of $\a$ with $h_{i}=Jx_i$, $\n_{\epsilon_i}=\R x_i$. Then $\tr \ad_{h_i}\neq 0$.
\end{prop}

\begin{proof}
As  $\a$ is Abelian the trace of $\ad_{h_i}$ is calculated just by considering an orthonormal basis $(e_i)$ of $\n$, and then $\tr\ad_{h_i}={\sum}_j\langle \ad_{h_i} e_j, e_j \rangle$. 
Suppose the element $e_j$ lives in the root space $\n_\alpha$. Then $\langle [h_i, e_{j}],e_{j}\rangle= \alpha(h_i)$ will not be zero if $ \alpha$ is $\epsilon_i$, $\tfrac{1}{2}\epsilon_{i}$, or $\tfrac{1}{2}(\epsilon_{i} \pm \epsilon_{k})$ with $1\leq i<k \leq n$. To be more precise, we have

\begin{longtable}[c]{c || c | c | c | c } 
$ \alpha$ &  $\epsilon_i$ & $\tfrac{1}{2}\epsilon_{i}$ & $\tfrac{1}{2}(\epsilon_{i}+\epsilon_{k})$ & $\tfrac{1}{2}(\epsilon_{i}-\epsilon_{k})$ \\ \midrule
$\langle [h_{i},e_{\alpha}],e_{\alpha} \rangle$ & 
$\epsilon_{i}(h_i)$ & $\tfrac{1}{2}\epsilon_{i}(h_i)$ & $\tfrac{1}{2}\epsilon_{i}(h_i)$ & $\tfrac{1}{2}\epsilon_{i}(h_i)$
\end{longtable}

Note that the values in the last two columns are the same. 
Equation \eqref{f(x_k)<0} says 
\begin{equation} \label{eq: traza}
\tr\ad_{h_i}=\sum_{j=1}^n \langle [h_i,e_j],e_j\rangle=\epsilon_{i}(h_i)+ \frac{1}{2}\epsilon_{i}(h_i)n_{i}+2\sum_{i<k}\frac{1}{2}\epsilon_{i}(h_i)n_{ik}> 0
\end{equation}
as dimensions are non-negative.
\end{proof}
\bigbreak

With all of the above in place we are ready to classify normal $j$-algebras of dimension $4$:

\begin{thm} \label{normal 4}
A four-dimensional normal $j$-algebra is isomorphic (as a Lie algebra) to either $\aff(\R)\times \aff(\R)$ or a semi-direct product $\R\ltimes \h_3$. More precisely, it belongs to one of two families of isomorphic Lie algebras:
\begin{itemize}
    \item[i)] $\d_{a,b}:=\op{span}\{h_1, x_1\}\oplus\op{span}\{h_2, x_2\}$, where $Jx_{1}=h_1$, $Jx_{2}=h_2$ and 
$$ [h_1, x_1]=a x_1, \quad [h_2,x_2]=b x_2, \qquad a,b> 0;$$
 \item[ii)] $\d_{a}:=\R h_1\ltimes \op{span}\{x_1 , x_2, x_3\}$, where $Jx_1=h_1 , \; Jx_3 = x_2$ and 
$$ [h_1,x_1]=ax_1, \quad
    [h_1,x_2]=\frac{a}{2}x_2, \quad
    [h_1,x_3]= \frac{a}{2}x_3, \quad
    [x_2, x_3]=ax_1, \qquad a> 0.$$
\end{itemize}
The basis in each case is orthonormal.
\end{thm}

\begin{proof}
As $\dim \a\leq \dim [\d, \d]$, the Abelian part can only be one- or two-dimensional, and we shall examine the occurrences separately.\par 
\quad If  $\a$ is a plane there are $2$ distinguished roots $\epsilon_1, \epsilon_2$, and $\n = \n_{\epsilon_1} \oplus \n_{\epsilon_2}$ is Abelian, by dimensional reasons. Pick orthonormal generators $x_i\in \n_{\epsilon_i}$, and define $h_i= Jx_i$, an orthonormal basis of $\a$, so 
$ [h_i, x_j]=\epsilon_j(h_i)x_j$. 
It follows from \eqref{anular raices} that $\epsilon_j(h_i)=0$ when the indices are different. The numbers $a:=\epsilon_{1}(h_1)$ and $b:=\epsilon_{2}(h_2)$ can be supposed positive in view of \eqref{f(x_k)<0}. 
Rescaling the basis of the two affine factors yields the product $\aff(\R)\times \aff(\R)$.
\par
\quad  If $\dim \mathfrak{a} = 1$ there exists a unique distinguished root $\epsilon$, and by Theorem \ref{prop raices} only one non-distinguished root $\frac12\epsilon$. Thus $\n= \n_{\epsilon} \oplus \n_{\epsilon/2}$ with $\dim \n_{\epsilon/2} =2$. 
    Choose an orthonormal basis for $\n$ such that $\n_{\epsilon}= \R x_1$ and $\n_{\epsilon/2}= \text{span}\{x_2, x_3\}$ with $x_2=Jx_3$ (recall $\n_{\epsilon/2}$ is $J$-invariant by Theorem \ref{prop raices}). As $[\n_{\epsilon/2},\n_{\epsilon/2}] \subseteq \n_{\epsilon}$, we write $[x_2,x_3]=b x_1$ for some $b \in \R$. Moreover, $x_1$ commutes with $x_2$ and $x_3$ since $\epsilon+\frac12 \epsilon$ is not a root. Setting $h_1:=Jx_1\in \a$, we have the following brackets:
$$
    [h_1,x_1]= ax_1, \quad
    [h_1,x_2]=\frac{a}{2}x_2, \quad
    [h_1,x_3]=\frac{a}{2}x_3, \quad 
    [x_2,x_3]=b x_1
$$
where $a = \epsilon_{1}(h_1) > 0$. Applying $f$ to the first and last relationships gives  $a f(x_1) = \omega(h_1,x_1) = 1 = \omega(x_2,x_3) = b f(x_1)$, hence $a=b$. We may finally adjust the metric to obtain the standard structure equations of $\h_3$ extended by $h_1$.
\end{proof}

\begin{rem}\label{rem: generalise}
{\rm 
It is not hard to generalise the family  $\d_a$ of Theorem \ref{normal 4} to produce normal $j$-algebras with one-dimensional Abelian part: for $n\geq 1$ and any real number $a>0$, let $\d_a^n$ be the $(2n+2)$-dimensional Lie algebra $\d_a^n=\a \oplus \n$, where 
\[\a=\R h_1, \quad \n=\text{span}\{x_1,u_1,\ldots,  u_{2n}\},\]
and the Lie bracket is 
 \begin{equation*}
      [h_1,x_1]=ax_1,\quad 
     [h_1,u_j]=\frac{1}{2}au_j, \quad 1\leq j\leq 2n,\qquad
     [u_{2i-1},u_{2i}]=ax_1, \quad 1 \leq i \leq n.
\end{equation*}
By declaring the basis $\{h_1,x_1,u_1,\ldots, u_{2n}\}$ orthonormal, and then setting $Jx_1=h_1$ and $Ju_{2i}=u_{2i-1}$ for $1\leq i\leq n$, this Lie algebra becomes a normal $j$-algebra. It is almost nilpotent, with nilradical isomorphic to the Heisenberg algebra $\h_{2n+1}$. 
For $n=1$ we recover the algebra $\d_a$ of the previous theorem. \par
\quad Furthermore, it is quite easy to check that any normal $j$-algebra with $1$-dimensional Abelian part is equivalent to one of our $\d_a^n$.}
\end{rem}

\bigbreak

\begin{prop} \label{ f se anula}
    Let  $(\d,J,f)$ be a  normal $j$-algebra. If $\alpha$ is a non-distinguished root then $f(\n_{\alpha})=\{0\}$, whereas if $\alpha=\epsilon_i$ is distinguished $f(\n_{\epsilon_i}) \neq \{0\}$ for all $i$.
 \end{prop}

\begin{proof}
Let $\epsilon_1,\ldots,  \epsilon_r$ denote the distinguished roots. A non-distinguished $\alpha$ must be $\tfrac{1}{2}\epsilon_{i}$ or $\tfrac{1}{2}(\epsilon_{i}\pm \epsilon_{j})$. In the latter case $f(\n_{\alpha})=0$ \cite{D}.  
If $\alpha= \tfrac{1}{2}\epsilon_{i}$ let us take an element $y \in \n_{\epsilon_i/2}$. If $\n_{\epsilon_i} =\R x_i$ then $[Jx_i, y]=\tfrac{1}{2}\epsilon_{i}(Jx_i)y$. On the other hand, 
\[ \tfrac{1}{2}\epsilon(Jx_i)f(y)= f\big([Jx_i, y]\big)=-df(Jx_i, y)=\omega(Jx_i, y)=\langle Jx_i, Jy \rangle=\langle x_i ,y \rangle=0  \]
since $x_i, y$ belong to different root spaces. Hence 
$f(y)=0$ due to \eqref{anular raices}, implying $f(\n_{\epsilon_i/2})=0$. Finally, for $\n_{\epsilon_i}=\R x_i$ it follows from \eqref{f(x_k)<0} that $f(x_i)\neq 0$.
\end{proof}

\vspace{.5cm}

\subsection{Einstein normal j-algebras}

The Einstein equation for normal $j$-algebras $(\d, J, f)$ is encoded into the condition, proved in \cite[Section 1]{Dotti}, that there exists a constant $C>0$ such that: 
$$
    C=\frac{\epsilon_{k}(Jx_k)}{f(x_k)}\left(1+\frac{1}{4}n_{k}+\frac{1}{2}\sum_{j<k}n_{jk}+\frac{1}{2}\sum_{j>k}n_{kj}\right)\qquad      \text{ for all }\ 1\leq k \leq r,
$$
 where $r=\dim\a$ and the numbers $n_{k},\ n_{ij}$ are as in  \eqref{eq: dimensiones}. If we choose the  $x_k$ to be unit vectors, equation \eqref{base ortonormal} forces $\epsilon_{k}(Jx_k)=\frac{1}{f(x_k)}$, and the previous equation becomes the slightly simpler, but much more practical condition
\begin{equation}\label{einstein en base ortonormal}
    f(x_k)^{2}=C^{-1} \left(1+\frac{1}{4}n_{k}+\frac{1}{2}\sum_{j<k}n_{jk}+\frac{1}{2}\sum_{j>k}n_{kj}\right), \qquad k=1,\ldots,  r.
\end{equation}

\smallskip

\begin{coro} \label{normal siempre einstein} 
A normal $j$-algebra $\d=\a\oplus\n$ with $\dim \a=1$ is Einstein.
\end{coro}

\begin{proof}
There is a unique distinguished root  $\epsilon$. If $\Delta'$ is empty  then $\d=\aff(\R)$, which is Einstein. 
If there do exist non-distinguished roots, by Theorem \ref{prop raices} there is actually only one, namely $\tfrac{1}{2}\epsilon$. In the above formula 
we then have  $r=1$, $n_{jk}=n_{kj}=0$ and $\epsilon_1(Jx_1)>0, \, f(x_1)>0$ by \eqref{f(x_k)<0}, and the Einstein condition reduces to the quantity $\frac{\epsilon_{1}(Jx_1)}{f(x_1)}(1+\frac{1}{4}n_1)$ being positive.
\end{proof}

\smallskip

\begin{rem}\label{rem: einstein dim=1}
{\rm 
A consequence of Corollary \ref{normal siempre einstein} is that the normal $j$-algebras $\d_a^n$ of Remark \ref{rem: generalise} are Einstein.}
\end{rem}
\medbreak

As a further example/application let us establish the following.
\begin{prop} The only Einstein normal $j$-algebra of dimension $2n$ with $\dim \a=\dim \n$ is the $n$-fold product $\aff(\R)\times \cdots \times \aff(\R)$.
\end{prop}
\begin{proof}
By Theorem \ref{prop raices} all roots $\{\epsilon_{1},\ldots, \epsilon_{n}\}$ are distinguished, and if each root space $\n_{\epsilon_i}$ is generated by the unit vector $x_i$, then $\{x_1, \dots , x_n \}$ is an orthonormal basis of $\n$, while $\{h_1,\ldots, h_n\}$ is an orthonormal basis of $\a$, where $h_i=Jx_i$. 
Then the bracket will be 
\[[h_i, x_i]=a_i x_i, \quad 1\leq i \leq n,\]
with all $a_i = \epsilon_{i}(h_i)\neq 0$. 
At the same time $\D x^{i}=-a_i h^{i}\wedge x^{i}$ says that the K{\"a}hler form 
\[ \omega :=h^{1}\wedge x^{1}+\cdots+h^{n}\wedge x^{n} = 
-\D \bigg( \sum_i \frac{1}{a_i}x^i\bigg)\]
is exact. D'Atri's theorem  ensures there is an Einstein metric precisely if  $\frac{a_k}{f(x_k)}=a_{k}^2=C>0$ for every $k$, that is, for $a_{1}=\cdots=a_{n}>0$. By rescaling, all these numbers can be taken to be $1$. 
\end{proof}

The next example provides another family of Einstein normal $j$-algebras.

\begin{example} 
{\rm    Let $(e_i)$ be an orthonormal basis of 
    $$\d=\a \oplus \n = \text{span}\{e_1,e_2\} \oplus \text{span}\{e_3, \ldots,  e_8\}.$$ 
We have two distinguished roots $\epsilon_1, \epsilon_2$, and let us suppose  the root spaces are 
$$\n_{\epsilon_1}=\R e_3,\quad \n_{\epsilon_2}=\R e_4,\quad 
\n_{(\epsilon_{1}-\epsilon_{2})/2}=\text{span}\{e_5,e_6\}, \quad
\n_{(\epsilon_{1}+\epsilon_{2})/2}=\text{span}\{e_7,e_8\},$$ 
with $Je_3=e_{1}, Je_{4}=e_2, Je_5=e_7, Je_6=e_8$. 
Equation \eqref{anular raices} tells that $\epsilon_{1}=a e^{1}, \epsilon_{2}=be^{2}$ with $a,b \neq 0$, and Theorem \ref{prop raices} imposes the following: 
$$\begin{array}{c}
[e_1,e_3]=ae_3 \quad
[e_1,e_5]=\tfrac{1}{2}ae_5 \quad
 [e_1,e_6]=\tfrac{1}{2}ae_6 \quad
 [e_1,e_7]=\tfrac{1}{2}ae_7 \quad
 [e_1,e_8]=\tfrac{1}{2}ae_8\\[2mm]
 {}[e_2,e_4]=be_4 \quad 
 [e_2,e_5]=-\tfrac{1}{2}be_5 \quad
 [e_2,e_6]=-\tfrac{1}{2}be_6 \quad
 [e_2,e_7]=\tfrac{1}{2}be_7 \quad
 [e_2,e_8]=\tfrac{1}{2}be_8 \\[2mm]
{}[e_5,e_7]=ce_3 \quad 
 [e_5,e_8]=ge_3 \quad
 [e_6,e_7]=he_3 \quad
 [e_6,e_8]=ke_3.
\end{array}$$
for some coefficients to be determined. Applying $f$ to $[e_5,e_8]=ge_3$ gives $g f(e_3)= f[e_5,e_8]= \omega(e_5,e_8)=\langle e_5, -e_6 \rangle=0$, so by \eqref{f(x_k)<0} we obtain $g=0$. 
By completely similar computations one finds $h=0$ and $c=k=-a$, so eventually $\d$ has equations
\begin{equation*}
\big(0, 0, a(-e^{13}+e^{57}+e^{68}), -be^{24}, 
\tfrac{1}{2}(be^{2}-ae^{1})e^5, 
\tfrac{1}{2}(be^{2}-ae^{1})e^6, 
\tfrac{1}{2}(be^{2}-ae^{1})e^7, 
-\tfrac{1}{2}(ae^{1}+be^{2})e^8\big).
\end{equation*}
The K{\"a}hler form $\omega=e^{13}-e^{42}-e^{57}-e^{68}$ 
can be written as $\omega=-\D f$ with
$f=\frac{1}{a}e^{3}+\frac{1}{b}e^{4}$. \par
\quad All in all, there is a $2$-parameter family of normal $j$-algebras $\d$ 
with $r=2, n_{1}=n_{2}=0,  n_{12}=2 $.
If we take $x_1=e_3, x_2=e_4$ in equation \eqref{einstein en base ortonormal}, then $\d$ is Einstein if and only if  $f(e_3)^2=f(e_4)^2$. This happens for  $a=\pm b$, so we end up with two curves of Einstein normal $j$-algebras within the big family:
\begin{equation*}\begin{array}{c}
\a\oplus \big(a(e^{57}+e^{68}-e^{13}), -ae^{24}, 
\tfrac{a}{2}(e^{25}-e^{15}), 
\tfrac{a}{2}(e^{26}-e^{16}), 
\tfrac{a}{2}(e^{27}-e^{17}), 
-\tfrac{a}{2}(e^{18}+e^{28})\big), \\[1mm]
\a\oplus \big(a(e^{57}+e^{68}-e^{13}), ae^{24}, 
-\tfrac{a}{2}(e^{25}+e^{15}), 
-\tfrac{a}{2}(e^{26}+e^{16}), 
-\tfrac{a}{2}(e^{27}+e^{17}), 
-\tfrac{a}{2}(e^{18}-e^{28})\big).
\end{array}\end{equation*}}
\end{example}

\bigbreak
\bigbreak

Next we show how to produce concrete examples of $\eta$-Einstein Sasakian Lie algebras beginning with certain normal $j$-algebras. We retain the notation of  \S \ref{sec:centreless} and Remark \ref{rem: generalise}. 

\begin{thm}\label{thm: examples with H0 no nulo}
Let $\k$ be a K{\"a}hler-Einstein Lie algebra with exact K{\"a}hler form and 
Einstein constant $c<0$, and $\d_a^n$ a $(2n+2)$-dimensional normal $j$-algebra as in Remark \ref{rem: generalise}, for some $n\in\N$, $a>0$. Set $\h_1 = \d_a^n\times\k$. If $H_0=-ax_1$ (and $k\neq 0$ is arbitrary), the associated centreless Sasakian Lie algebra $\g$, constructed using Proposition \ref{prop: compatibilidad}, is $\eta$-Einstein if and only if \[c^{2}=\frac12 a^2(n+3).\]
This structure is negative $\eta$-Einstein with constant $\lambda=-c^2-2$.
\end{thm}

\begin{proof}
Corollary \ref{normal siempre einstein} and Remark \ref{rem: einstein dim=1} guarantee that the K{\"a}hler metric on $\d_a^n$ is Einstein. 
Moreover, it is straightforward to check that the Ricci curvature equals 
\begin{equation}\label{Ricci en normal con 1 raiz}
    \ric^{\d_a^n}= -a^2\left(\frac{n}{2}+1\right) \pint_{\d_a^n},
\end{equation}
and so $\ric^{\h_1}$ has block-diagonal form 
$\ric^{\d_a^n} \oplus\, c\Id_{2n}$ in terms of the orthogonal sum $\h_1=\d_a^n\oplus \k$. Therefore, equation \eqref{Ric en h1} gives 
\[ \ric (h_1,h_1)=-a^{2}-\frac{1}{2}a^{2}n-2+\frac{1}{2}\langle \theta(Jh_1,h_1), -ax_1\rangle=-\frac{1}{2}a^{2}(n+3)-2\]
and $\ric (x_1,x_1)=\ric (h_1,h_1)$, as $\ric(\varphi x,\varphi y)=\ric (x,y)$ whenever $x,y\perp \xi$. A direct computation gives the same value for $\ric (u_i, u_i)$, $1\leq i \leq 2n$,  
and also for $\ric(u,u)=\ric (v, v)=-a^{2}-2+\frac{1}{2}-a\tr\ad_{h_1}$ 
with $u,v\in {\q}$.  Taking a unit vector $x\in \k$, on the other hand, produces
\[ 
\ric (x, x)=-c^{2}-2+\frac{1}{2}\langle \theta(Jx, x), -ax_1 \rangle=-c^{2}-2.
\]
So, for $\g$ to be $\eta$-Einstein the numbers must match, which happens precisely when $2c^{2}=a^{2}(n+3)$. At last, $-c^2-2<-2$ says the $\eta$-Einstein structure is negative.
\end{proof}

\bigbreak 

\section{Modifications}

One of the core lessons we have learnt thus far is that a good chunk of the theory of $\eta$-Einstein Sasakian algebras boils down, so to speak, to the understanding of K{\"a}hler-exact solvable algebras. This subject has of course its own interest, irrespective of our Sasakian construction on top of it. 

In this part we shall review, and elaborate on, the notion of modified Hermitian algebra as presented in \cite{Dorf, Dalek-Keizo-Kamishima}. 
Modifications, besides being a very flexible and practical algebraic tool, are a core aspect of the Gindikin--Vinberg conjecture on homogeneous K\"ahler manifolds \cite{Dorfmeister-Nakajima}.

\begin{definition} \label{def modi}
Let $\big(\h,[\cdot,\cdot ], J, \pint\big)$ be a Hermitian Lie algebra. A linear map 
 $\phi: \h \to \op{Der}_{u}( \h)$ satisfying 
$$[\phi(x), \phi(y)]=0, \qquad \phi [x,y]=0, \qquad \phi\big(\phi(x)y\big)=0$$
for all $x,y \in \h$ is  is called {\bf modification map}.
\end{definition}
Purely for notational purposes we shall `curry' a modification $\phi$ and understand it as a family of unitary derivations $\big\{\phi_x \in\op{Der}_u(\h)\colon \phi_x(y):=\phi(x)y,\ x, y\in \h\big\}$ parametrised by $\h$ (more prosaically, we just enforce the exponential law).
\smallbreak

 Define a modified bracket in the following way:
\[[x,y]_{\phi}:=[x,y]+\phi_xy-\phi_y x.\]
This is indeed a Lie bracket \cite[Proposition 5]{Cortes}, and we shall indicate the new structure by $\h_{\phi}:=(\h,[ \cdot , \cdot]_{\phi})$, and write $\D^\phi$ for its Chevalley-Eilenberg differential.
\medbreak
This notion starts to be significant from dimension four upwards, since a modification of a 2-dimensional Lie algebra is trivial. Indeed, for $\R^2$, take $e_1, e_2$ orthonormal  and let $\phi(v)=
    \begin{spmatrix} 0 & -g(v)\\ g(v) & 0  \end{spmatrix}$ 
be a modification, for some linear map $g$. Then $0= \phi(\phi(e_1)(e_1))=\phi(g(e_1)e_2)=g(e_1)\phi(e_2)$, and similarly $g(e_2)=0$. The proof for the non Abelian $2$-dimensional algebra, $\aff(\R)$, follows immediately from the definition of modification map.
     
\vspace{.5cm}

Henceforth we shall be concerned with solvable K{\"a}hler algebras $\s$. Dorfmeister showed that any modification $\s_{\phi}$ stays solvable K{\"a}hler. 
He actually proved the following theorem:
\begin{prop}[\cite{Dorf}]\label{dorf} 
    Every solvable K{\"a}hler Lie algebra $(\s, J)$ is a modification of a completely solvable semi-direct product of a normal $j$-algebra times an Abelian ideal, both of which are $J$-invariant.
\end{prop}

Because of this result, our focus can shift to modified normal $j$-algebras.  Specifically in dimension four, for instance, we have:

\begin{prop} \label{clasi dim 4}
    A $4$-dimensional, modified normal $j$-algebra $\s_{\phi}$ is either isomorphic to
\begin{itemize}
    \item[i)] one of the $\d_{a,b}$ in Theorem \ref{normal 4}, or 
\item[ii)] a modification of 
one of the $\d_{a}$ in Theorem \ref{normal 4}, with structure equations  
$$[h_1,x_1]= a x_1,\quad
    [h_1, x_2]=\frac{1}{2}ax_2+ cx_3,\quad
    [h_1, x_3]=\frac{1}{2}ax_3 - cx_2,\quad
    [x_2, x_3]=x_1,
$$ 
where the additional parameter $c\in\R$ accounts for the modification.
\end{itemize}    
\end{prop}

\begin{proof}
  i)\ Let $\phi$ be a modification of $\d_{a,b}$. Lemma \ref{Da=0} warrants that  
  it suffices to consider  $[\a,\n]_{\phi}$. As 
  $\dim \a= \dim \n$, by Theorem \ref{prop raices} the only two roots $\epsilon_{1}, \epsilon_{2}$ are  distinguished. 
 Now,  $\phi_\a (\n)=0$ and $\phi_\n(\a)=0$ due to Lemmas 
\ref{Da=0}--\ref{D cero}, resulting in  $[\a,\n]_{\phi}=[\a,\n]$.
Hence any modification acts trivially on $\d_{a,b}$. \par
\quad ii)\ In view of Lemmas \ref{Da=0}--\ref{D invariante}, the modification acts on a vector $v \in \d_a$ as 
\begin{equation*}
\phi(v)=
    \begin{spmatrix}
        0 & 0 & 0 & 0\\
        0 & 0 & 0 & 0 \\
        0 & 0 & 0 & -g(v) \\
        0 & 0 & g(v) & 0
    \end{spmatrix}
\end{equation*}
for some linear functional $g$. Letting  $c=g(h_1)$, we recover the structure equations claimed.
\end{proof}

\begin{prop} \label{conmutador modi igual}
Any modification $\phi$ of a  normal $j$-algebra $\d$ does not alter the derived algebra: $[\d,\d]_{\phi}=[\d,\d]$.
\end{prop}

\begin{proof}
$(\subseteq)$ \ First of all, $[\a,\a]_{\phi}=[\a,\a]$ (Lemma \ref{Da=0}) and clearly $[\n,\n]_{\phi}=[\n,\n]$. 
Now take $x\in \a$, $y=\sum_{\alpha \in \Delta}y_{\alpha} \in \n=\bigoplus_{\alpha \in \Delta} \n_{\alpha}$, so 
$$[x,y]_{\phi}- [x,y] =\phi_x y = \sum_{\alpha \in \Delta}\phi_x y_{\alpha}.
$$
This combination belongs in $\n$ by Lemmas \ref{D invariante}--\ref{D suma}, proving the claim. \par
\quad $(\supseteq)$ \  Let $\alpha=\epsilon_i\neq 0$. As 
$[\a, \n_{\epsilon_i}]_{\phi}=[\a, \n_{\epsilon_i}]+\phi_\a \n_{\epsilon_i}-\phi_{\n_{\epsilon_i}} \a=[\a, \n_{\epsilon_i}]=\epsilon_{i}(\a)\n_{\epsilon_i}$, we deduce $\n_{\epsilon_i} \subseteq [\d,\d]_{\phi}$ and therefore $\bigoplus \n_{\epsilon_i} \subseteq [\d,\d]_{\phi}$.\par
\quad Taking $\alpha=\epsilon_i/2$ it follows that $\phi_h\colon \n_{\epsilon_i/2} \to \n_{\epsilon_i/2}$ commutes with  $J\restr_{\n_{\epsilon_i/2}}$ and both are skew. There is an orthonormal basis  $\{e_j,f_j \}$ of $\n_{\epsilon_i/2}$ and linear forms $\gamma_j\in \a^*$ ($j=1,\ldots,  d$) such that
\[  J \vert_{\n_{\epsilon_i/2}}=\begin{spmatrix}
   0 & -1 &  & &\\ 
   1& 0 &  & &\\ 
   &  &  \ddots & &\\
   &  & &  0 & -1\\
   &  & & 1 & 0 
 \end{spmatrix} 
\qquad 
  \phi_h  \vert_{\n_{\epsilon_i/2}}=\begin{spmatrix}
   0 & -\gamma_{1}(h) &  & & \\ 
   \gamma_{1}(h)& 0 &  & &\\ 
   &  &  \ddots & & \\
   &  &  & 0 & -\gamma_{d}(h)\\
   &  &  & \gamma_{d}(h) & 0 
 \end{spmatrix}\]
can be written in block form. This tells us that
\[Je_k=f_k, \quad  \phi_h e_k= \gamma_{k}(h)f_k, \quad  \phi_h f_k=-\gamma_{k}(h)e_k,\]
so 
$[h,e_k]_{\phi}=\tfrac{1}{2}\epsilon_{i}(h)e_k+ \gamma_{k}(h)f_k$,\ $
    [h, f_k]_{\phi}= \tfrac{1}{2}\epsilon_{i}(h)f_k-\gamma_{k}(h)e_k$, 
or more compactly 
\begin{equation*}
    \ad^{\phi}_{h}\restr_{\text{span}\{e_k, f_k\}}= \begin{pmatrix}
        \frac{1}{2}\epsilon_{i}(h) & -\gamma_{k}(h) \\
        \gamma_{k}(h) & \frac{1}{2}\epsilon_{i}(h)
    \end{pmatrix}.
\end{equation*}
Note that there exists an element $h \in \a$ with $\epsilon_{i}(h) \neq 0$, and so the above matrix' determinant $\frac{1}{4}\epsilon_{i}^{2}(h)+\gamma_{k}^{2}(h)$ is different from zero. 
Therefore $\ad_{h}^{\phi}$ is an  isomorphism on $ \text{span} \{e_k, f_k\}$,  and thus $\n_{\epsilon_{i}/2}\subseteq \text{span}\{e_k, f_k\} \subseteq \text{Im} \ad_{h}^{\phi} \subseteq [\d, \d]_{\phi}$. \par
\quad For the last type of non-distinguished root, observe that $\n_{ij}:=\n_{(\epsilon_{i}-\epsilon_{j})/2}\oplus \n_{(\epsilon_{i}+\epsilon_{j})/2}$ is both $J$- and $\phi$-invariant (Theorem \ref{prop raices} and Lemma  \ref{D suma}). Choose orthonormal bases 
 $\{ u_1, u_2, \ldots,  u_r\}$ and $\{v_1,v_2,\ldots, v_r\}$ in the two spaces, with $Ju_i=v_i$. Since $[\phi_h, J]=0$ for any $h\in \a$, and either is skew, the overall basis  $\{u_i, v_j \}$ is orthonormal for $\n_{ij}$ and yields the normal forms
\begin{equation*}
J\restr_{\n_{ij}}=
\left(\begin{array}{c|c}
      & -\operatorname{Id}_{r\times r}\\ \hline
     \operatorname{Id}_{r \times r}&  
\end{array} \right)
\qquad 
\phi_h \restr_{\n_{ij}}=
\left( \begin{array}{c|c}
      A(h)& -B(h)\\ \hline
    B(h) &  A(h)
\end{array} \right)\in\U(r).
\end{equation*}
Write $\n_{\epsilon_i}=\mathbb{R}x_i$, and apply the above to $h_i=Jx_i\in\a$ to find
$$
    [h_i, u_k]_{\phi}=\frac{1}{2}\epsilon_{i}(h_i)u_k+\phi_{h_i} u_k\qquad 
    [h_i, v_k]_{\phi}=\frac{1}{2}\epsilon_{i}(h_i)v_k+ \phi_{h_i} v_k
$$
where $\epsilon_{i}(h_i)=\lambda_{i}\neq 0$. 
Let us consider the operator $\ad^{\phi}_{h_i}\restr_{\n_{ij}}= \lambda_{i}\Id\restr_{\n_{ij}} + \phi_{h_i}\restr_{\n_{ij}}$: we claim its eigenvalues are all non-zero. Indeed, $\phi_{h_i}$ has purely imaginary spectrum for being unitary, and we just saw that $\lambda_i\neq 0$. Hence the eigenvalues of $\lambda_i \Id+\phi_{h_i}$ have non-zero real part. \par
In consequence of that, 
$\n_{ij} \subseteq \text{Im} \ad^{\phi}_{h_i}\subseteq [\d, \d]_{\phi}$ and  $\bigoplus \n_{ij} \subseteq [\d,\d]_{\phi}$. \par
To recap, we have shown that the sums 
${\textstyle\bigoplus}\n_{\epsilon_i}$, $
{\textstyle\bigoplus} \n_{\epsilon_i/2}$ and ${\textstyle\bigoplus} \n_{ij}$  all live in $[\d ,\d]_{\phi}$, 
so $[\d,\d]= \bigoplus_{\alpha \in \Delta} \n_{\alpha} \subseteq [\d, \d]_{\phi}$.
\end{proof}

\begin{prop} 
For any modification $\phi$ of a normal $j$-algebra, $ \tr \ad_{h}=\tr \ad ^{\phi}_{h}$ for all $h \in \a$.
\end{prop}

\begin{proof} Choose an orthonormal basis $(e_i)$, so $\tr\ad_{h}^{\phi} =\sum_i  \langle [h,e_i]_{\phi},e_i \rangle$. Lemma \ref{Da=0} ensures that $[h,e_i]_{\phi}=[h,e_i]+\phi_h e_i - \phi_{e_i}h=[h,e_i]+ \phi_h e_i$, 
and the last summand does not really count as $\phi_h$ is skew. 
\end{proof}

\subsection{Compact quotients of Sasakian Lie groups}

As is already patent in definition \ref{def modi} the complex structure $J$ plays no part in defining the Lie-algebraic structure of the modification. We required $\phi$ to be unitary only for the purposes of adapting the resulting structure to the geometry of normal $j$-algebras, which are the main concern. But the concept makes sense on any metric Lie algebra $\big(\g, \pint\big)$, where a modification is a homomorphism
 \[\phi\colon\g \longrightarrow \op{Der}(\g) \cap\so(\g)\]
 such that $\phi ([\g,\g])=0$ and $\phi\big(\Im \phi(\g)\big)=0$.
Cort{\'e}s and Hasegawa  \cite{Cortes} had already observed this by showing 
that a unimodular Sasakian Lie algebra is either $\su(2)$, 
 $\slie(2,\R)$ (simple) or a modification $(\h_{2n+1})_{\phi}$ of the Heisenberg Lie algebra (solvable). It then follows that a unimodular Sasakian Lie algebra of dimension $\geq 5$ is necessarily solvable. 

\begin{lemma}
Let $\h_{2n+1}$ be the $(2n+1)$-dimensional Heisenberg Lie algebra equipped with an inner product $\pint$, and $\phi$ a modification of it. Then $(\h_{2n+1})_{\phi}$ has non-trivial center. 
\end{lemma}

\begin{proof}
Consider an orthonormal basis $\{x_1, \ldots,  x_n, y_1, \ldots,  y_n, z \}$ of $\h_{2n+1}$ giving  brackets $[x_i, y_i]=r_i z$ 
for some $r_i>0$, for all $i=1,\ldots,  n$. The element $z$ spans the centre. 
Take $x \in (\h_{2n+1})_{\phi}$, so that $[x,z]_{\phi}=[x,z]+\phi_x z-\phi_z x = \phi_x z$.  
As derivations preserve the derived algebra, we have $\phi_x z=c z$ for some $c\in \R$. This, together with the fact that $\phi_x $ is skew, tells that 
$0=\langle \phi_x z, z\rangle=\langle cz, z \rangle = c\vert z \vert ^{2}$.
Eventually, $z$ is a central element in $(\h_{2n+1})_{\phi}$.
\end{proof}

\smallskip

\begin{thm}\label{thm: lattices}
   Let $G$ be a Lie group of dimension $\geq 5$ with a left-invariant $\eta$-Einstein Sasakian structure. If it admits lattices, then $G$ is solvable and the structure is null $\eta$-Einstein. 
\end{thm}
    
\begin{proof} 
The existence of lattices in $G$ is known to guarantee unimodularity (see \cite{Raghunathan} or \cite{Mil}, for instance). The aforementioned Cort{\'e}s--Hasegawa theorem says the Lie algebra $\g$ must be a modification of $\mathfrak{h}_{2n+1}$. The previous lemma warrants non-trivial center, and by Proposition \ref{g unimodular} the structure has to be null.
\end{proof}

\smallskip

Immediately then, 
\begin{coro}
Every $\eta$-Einstein solvmanifold (of dimension at least $5$) is null.
\end{coro}

\medskip

\section{K{\"a}hler-exact Lie algebras}

We will retain all earlier conventions, but now denote K{\"a}hler-exact Lie algebras by $\k$. 

\begin{lemma} \label{exacta no poseen abeliana}
A K{\"a}hler-exact Lie algebra has no (non-trivial) $J$-invariant Abelian subalgebra.
\end{lemma}

\begin{proof}
Assume the K{\"a}hler form is the differential of some functional $f \in \k^*$. 
If there existed a non-trivial $J$-invariant Abelian subalgebra $\a$, we could pick an element $ 0\neq x \in \a$ and then  $0=f[x,Jx]=-\D f(x,Jx)=-\omega(x,Jx)=\|x\|^2$, a contradiction. 
  \end{proof}

\begin{prop} \label{modi preserva exacta}
    Let $(\d, J, f)$ be a  normal $j$-algebra. Any modification $\d_{\phi}$ is a 
    K{\"a}hler-exact Lie algebra with fundamental form $\omega=-\D^{\phi}f$.
\end{prop}

\begin{proof} We shall examine $\D^{\phi}f(x,y)= -f\big([x,y]+\phi_x y-\phi_y x\big)$ for all possible $x,y \in \d=\a\oplus \n$. 
 If $x,y \in \a$ the bracket does not change (Lemma \ref{Da=0}) and therefore $\D^\phi f=\D f$. The same goes for $x,y \in \n$, due to the definition of modifications. The slightly more involved situation is when $x \in \a$, $y \in \n$, in which case we  have already seen that  
$$
        [x,y]_{\phi}=[x,y]+\phi_x y.
$$
Using the root decomposition we might as well assume $y\in \n_\alpha$ for some $\alpha\in \Delta$, which may be distinguished ($\alpha=\epsilon_i$ for some $1\leq i\leq r$) or not ($\alpha \in \Delta'$). In the former case Lemma \ref{D cero} implies $\phi_x y=0$, so again $\D^{\phi}f=\D f$.  
In the latter, $\phi_x y\in \bigoplus_{\beta \in \Delta'}\n_{\beta}$ by Lemmas \ref{D invariante}, \ref{D suma}. According to Proposition \ref{ f se anula} $f$ vanishes on $\n_\beta$ for any such $\beta$, and therefore $f[x,y]_\phi=f[x,y]$. 
This proves the last case and ends the proof.
\end{proof}

\smallskip

The good news is that we can reverse the above result. 

\begin{thm} \label{thm: exacta modi normal}
A K{\"a}hler-exact solvable Lie algebra is a modification of a normal $j$-algebra.
\end{thm}

\begin{proof} Call $\k$ the Lie algebra in question. We know from Proposition \ref{dorf} that $\k$ is a modification $\g_{\phi}$, where $\g=\d\ltimes \b$ is completely solvable K{\"a}hler and obtained as semi-direct product of a normal $j$-algebra $\d$ and an Abelian ideal $\mathfrak{b}$ (both $J$-invariant). 

It follows from \cite[Section 3.3, Lemma 2]{Dorf} that $\phi_x$ leaves $\b$ invariant for any $x\in\g$, and this implies $\b$ is a subalgebra of $\g_\phi$. Moreover, the map $\phi_\b:\b \to \operatorname{Der}_u(\b)$ defined by $(\phi_\b)_x=\phi_x\restr_\b$ is a modification of $\b$. In particular, $\b\subseteq \g_\phi$ is isomorphic to the modification $\b_{\phi_\b}$. Applying now \cite[Section 3.3, Lemma 1]{Dorf} we obtain that $\b_{\phi_{\b}}=\b_0 \oplus \b_1$ is an orthogonal sum of two Abelian $J$-invariant subalgebras of $\b_{\phi_\b}$. Since $\b_{\phi_\b}$ is isomorphic to $(\b,[\cdot,\cdot]_\phi)$, both $\b_0$ and $\b_1$ must be $J$-invariant Abelian subalgebras of $\g_\phi=\k$. Yet we saw  in Lemma \ref{exacta no poseen abeliana} that this cannot happen, so the only chance is that both $\mathfrak{b}_0$ and $\mathfrak{b}_1$ are zero. So all in all $\mathfrak{b} = \{0\}$, and our initial $\k$ is indeed a normal $j$-algebra.
\end{proof}

\begin{prop} \label{kahler soluble}
Every $4$-dimensional K{\"a}hler-exact Lie algebra is solvable, and hence classified by Proposition \ref{clasi dim 4}.
\end{prop}
\begin{proof}
Straightforward consequence of the previous theorem.
\end{proof}

We saw in Proposition \ref{h1 crea sasaki sin centro} that the construction 
of Sasakian algebras from solvable K{\"a}hler-exact algebras  $\h_1$ depends in a crucial way upon how one picks a special element $H_0\neq 0$. Clearly, having too much freedom of choice is undesirable for the purposes of a classification. We show next that this is not the case, for the range of choices turns out to be actually finite.

\begin{prop}\label{prop: finite}
    Let $(\h_1,\theta)$ be a K{\"a}hler-exact  solvable Lie algebra, and suppose $H_0\in \h_1$ satisfies Proposition \ref{h1 crea sasaki sin centro}. Then either $H_0=0$, or $H_0$ belongs in a finite subset $F\subset \theta(\h_1,\h_1)$ of cardinality $|F|=\dim\h_1 / \theta(\h_1,\h_1)$.
\end{prop}
\begin{proof}
The algebra $\h_1$ is obtained modifying a normal $j$-algebra $(\d,[\cdot,\cdot])$ via a certain $\phi$, and any non-null $H_0$ will live in the derived space $\theta(\h_1,\h_1)$, hence in the nilradical $\n=[\d,\d]$ (Proposition \ref{conmutador modi igual}).
As $JH_0\in\a=\n^\perp$, Theorem \ref{prop raices} says that $H_0\in \bigoplus_{i=1}^r \n_{\epsilon_i}$, where $r=\dim \a$ and $\epsilon_1,\ldots,  \epsilon_r$ are the distinguished roots. Since $\theta$ equals the modified bracket $[\cdot ,\cdot ]_\phi$, equation \eqref{eq: importante} reads
$2\mu(x)H_0= [x,H_0]+\phi_x H_0-\phi_{H_0} x$. But the last two summands vanish ($H_0\in \n$ and Lemma \ref{D cero}), so we are left with
$2\mu(x)H_0=[x,H_0]$ for all $x\in \h_1$. That means $2\mu$ is a distinguished root $\epsilon_i$. Let us then write $H_0=tx_i$ for a unit generator $x_i\in\n_{\epsilon_i}$ and $0\neq t \in \R$. Since $-2\mu$ is the metric dual to $JH_0$, it follows that    $\epsilon_{i}(x)=-t\langle x, h_i \rangle$ 
with $h_i=Jx_i$. Taking $x=h_i$ now produces  $t=-\epsilon_i(h_i)$. Altogether $H_0$ belongs in the finite set 
\[\{-\epsilon_i(h_i)x_i \mid i=1,\ldots,  r\}\]
of cardinality $r=\dim \a=\dim\h_1-\dim \theta(\h_1,\h_1)$.
\end{proof}

\vspace{.5cm}

\subsection{Curvature} 

It is rather straightforward to check that a general modification $\h_{\phi}$ has  Levi-Civita connection 
$$\nabla^{\phi}_{x}=\nabla_{x}+\phi_x$$
(see \cite[Proposition 6]{Cortes} as well). 

\begin{prop}\label{prop: Ricci invariante}
Modifications preserve the entire Riemann tensor, and hence we have $\ric^{\phi}=\ric$ too.
\end{prop}
\begin{proof}
Take any $x, z \in \h$.  Then 
\begin{align*}
    R^{\phi}_{z, x}&=\nabla^{\phi}_{z}\nabla^{\phi}_{x}-\nabla^{\phi}_{x}\nabla^{\phi}_{z}- \nabla^{\phi}_{[z, x]_{\phi}}\\
&=\nabla_{z}\nabla_{x}+\phi_z \nabla_{x}+\nabla_{z}\phi_x +\phi_z\phi_x - \nabla_{x}\nabla_{z}-\nabla_{x}\phi_z  -\phi_x \phi_z \\
&\quad  -\nabla_{[z, x]}-\phi_{[z, x]}- \nabla_{\phi_z x}-\phi_{\phi_z x}+\nabla_{\phi_x z}+\phi_{\phi_x z}\\
&=R_{z, x}+\phi_z \nabla_{x}+\nabla_{z}\phi_x -\phi_x \nabla_{z}-\nabla_{x}\phi_z  -\nabla_{\phi_z x}  + \nabla_{\phi_x z}
\end{align*}
But any skew derivation $D$, such as our maps $\phi_x$, satisfies 
    $\nabla_{x}D +\nabla_{Dx}=D\nabla_{x}$.
This means that in the last line above only $R_{z, x}$ survives whilst the other terms cancel each other out. 
\end{proof}

\smallskip

Let us go back to the setting of \S \ref{sec:centreless}, where $\g$ was a centreless Sasakian Lie algebra obtained from a K{\"a}hler-exact Lie algebra $(\h_1,\theta)$ and some vector $H_0\neq 0$. 
Assuming $\h_1$ is solvable, it follows from Theorem \ref{thm: exacta modi normal} that $\h_1$ is a modification $\d_\phi$ of some normal $j$-algebra $\d$. Using this data, in formula \eqref{Ric en h1} we may replace the term $\ric^{\h_1}$ by $\ric^\d$ (Proposition \ref{prop: Ricci invariante}), and the summand $\langle \theta(Jx,y),H_0\rangle$ reads
\[ \langle \theta(Jx, y), H_0  \rangle=\langle [Jx, y]+\phi_{Jx} y-\phi_yJx,H_0\rangle
    =\langle [Jx, y],H_0\rangle+ \langle \phi_{Jx} y,H_0\rangle -\langle \phi_yJx,H_0 \rangle.
\]
In the right-hand side, the last two terms vanish because, for instance, 
$\langle \phi_{Jx} y,H_0\rangle=-\langle y, \phi_{Jx} H_0 \rangle=0$ (Lemma \ref{D cero}, $H_0 \in \n_{\epsilon_i}$). 
 Therefore the Ricci tensor reads 
\[\ric(x,y)=\ric^{\d}(x,y)-2\langle x, y\rangle +\frac{1}{2}\langle [Jx, y],H_0\rangle, \qquad x,y \in \h_1.\]
The fallout is that when studying Einstein-like conditions, it is enough to consider normal $j$-algebras only.
\medbreak

With that in place we now assume $\h_1$ is a normal $j$-algebra, and therefore all of \S \ref{sec:normal-j-algebras} applies. We have seen that $H_0$, if non-zero, is of the form $-\epsilon_{i}(h_i)x_i \in \n_{\epsilon_i}$ for some $i=1,\ldots,  r$ (Proposition \ref{prop: finite}). 
The following result exhibits a much needed necessary condition on $\h_1$ that obstructs the construction of $\eta$-Einstein Sasakian structures.

\begin{prop}\label{prop: restriction} 
Consider the centreless Sasakian solvable algebra $\g$ built from the normal $j$-algebra $\h_1$ and $0\neq H_0\in \n_{\epsilon_i}$. 
If $\g$ is $\eta$-Einstein then $n_{ji}=0$ for all $j<i$.
\end{prop}

\begin{proof}
As usual, let $\a$ denote the Abelian subalgebra orthogonal to the commutator $[\h_1, \h_1]$. We borrow a computation from \cite[Section 3]{D} to arrive at 
 \begin{equation}\label{ricci en a}
     \ric^{\h_1}(h, h)=-\sum_{\alpha \in \Delta} \alpha(h)^{2}\dim \n_{\alpha}, \qquad  \text{ for any }\ h\in \a.
 \end{equation}
 Consider the orthonormal basis $\{h_1, \ldots,  h_r\}$ of $\a$, where $h_i=Jx_i$ and the unit vector $x_i$ generates the distinguished root space $\n_{\epsilon_i}$. Then $\alpha(h_i)\neq 0$ if and only if $\alpha$ is one among $\epsilon_i$,\ $\frac 12 \epsilon_i$,\ $\frac{1}{2}(\epsilon_i \pm \epsilon_ j)$ with $i<j$, or $\frac{1}{2}(\epsilon_j \pm \epsilon_i)$ with $j<i$, and 
 \begin{equation}
     \ric^{\h_1}(h_i, h_i)=-c^2\left(1+\frac{1}{4}n_i+\frac{1}{2}\sum_{i<j}n_{ij}+\frac{1}{2}\sum_{j<i}n_{ji}\right)
 \end{equation}
 where $c:=\epsilon_i(h_i)$. 
 Using \eqref{Ric en h1} with $H_0=-cx_i$ we obtain
\begin{equation*}\begin{split}
     \ric(h_i,h_i)&=\ric^{\h_1}(h_i,h_i)-2+\frac{1}{2}\langle \nonumber \theta(Jh_i,h_i),H_0 \rangle \\ 
     &=\ric^{\h_1}(h_i,h_i)-2-\frac{1}{2}c^2
     =-c^2\left(\frac{3}{2}+\frac{1}{4}n_i+\frac{1}{2}\sum_{i<j}n_{ij}+\frac{1}{2}\sum_{j<i}n_{ji}\right)-2. 
     \end{split}\end{equation*}

 On the other hand, formulas \eqref{Ric en u,u}, \eqref{ad y traza} and \eqref{eq: traza} together force
\[ \ric(u,u)=-c^{2}\left(\frac{3}{2}+\frac{1}{4}n_i+\frac{1}{2}\sum_{i<j}n_{ij}\right)-2. \]
Eventually, if $\g$ is $\eta$-Einstein necessarily $\sum_{j<i} n_{ji}=0$. But as the $n_{ji}$ are non-negative, they must all vanish.
\end{proof}


\section[Centreless Sasakian algebras with $ \Ker \ad_{\xi}=\R \xi $]{Centreless Sasakian algebras with \texorpdfstring{$ \Ker \ad_{\xi}=\R \xi $}{}}

To complete the picture we have dealt with throughout the paper, we consider the extreme case, so to speak -- that being when $\Im  \ad_{\xi}$  is  largest possible (codimension one). 

\begin{prop}\label{prop: final}
A Sasakian Lie algebra with no centre and $\Ker \ad _{\xi}= \R\xi$ is three-dimensional and simple (that is, $\su(2)$ or  $\slie(2, \R)$).
\end{prop}

\begin{proof}
Start from the splitting $\g = \R \xi \oplus \h$. As $\ad_\xi$ preserves $\h$, 
its restriction ${\ad_{\xi}}\restr_{\h }$ is an isomorphism, so 
    $\h= \ad _{\xi} ( \h)  \subseteq [\g , \g]$.
But $\varphi$ preserves $\h$, too, ie 
\begin{equation}\label{[x,varphi x]}
    [x,\varphi x]=-2\omega(x,\varphi x)\xi + \theta(x,\varphi x)=2\|x\|^2 \xi + \theta(x,\varphi x)
\end{equation}
for any $x\in\h$, where $\theta$ is the usual $\h$-component of $[\cdot,\cdot]$. This implies $\xi \in [\g,\g]$ whence 
$$\g = [\g, \g]$$ 
 is perfect. It may be semisimple or not. In the former case $\g$ must be $\su(2)$ or $\slie(2, \R)$, as these are the only semisimple algebras admitting a contact form  \cite[Theorem 5]{B}.\par
\quad The rest of the proof will show that there is no latter case, that is, $\g$ must be semisimple. 
Take the Levi decomposition $\g = \s \ltimes \r$, where $\s$ is a semisimple subalgebra and $\r$ the radical. Let $\n$ be the nilradical of $\g$ (which is the nilradical of $\r$). Any derivation of the solvable algebra $\r$ sends it into $\n$. Applying this fact to the ${\ad_x}\restr_{\r}$, with $x\in \s$, we obtain $[\s, \r] \subseteq \n$, hence  
\[  [\g, \g]= [\s, \s] \oplus \big([ \s, \r] + [\r, \r]\big) \subseteq \s \oplus \n. \]
The fact that $\g$ is perfect implies $\n=\r$. \par
\quad Since $\n$ is an ideal, the adjoint map ${\ad_{\xi}}\restr_{{\n}}\colon \n \to \n$ is a well-defined endomorphism.  We claim it is an isomorphism.
Indeed, if there existed $x \in \n$ such that  $[\xi, x] =0$ then $x = c\, \xi$ for some $c \in \R$ by hypothesis.  Supposing $c\neq 0$ would imply $\xi \in \n$, making $\ad_{\xi}$ a nilpotent endomorphism. At the same time $\ad_\xi$ is skew, so 
$\xi$ would have to be central, which cannot be. Hence $c=0$ and  $x=0$.  \par
\quad Now that ${\ad_{\xi}}\restr_{\n}$ is an isomorphism, we have $\n \subseteq \Im \ad_{\xi}=\h$. 
Consider any $x \in \n$, so from \eqref{[x,varphi x]} 
$$
2\|x\|^2 \xi =[x ,\varphi x] - \theta (x, \varphi x),
$$
with $\theta (x, \varphi x)\in \h$. The right-hand side lives in $\h$ since $\n$ is an ideal contained in $\h$. So, if $x\neq 0$ then   $\xi \in \h$, and at the same time $\xi \perp \h$, which is a contradiction. Therefore $\n=0$ and $\g$ is semisimple. 
\end{proof}

Of course the above proposition can be read off \cite{Cortes} since perfect implies unimodular. Our proof has a flavour of its own and ties in with the earlier discussion.

\bigbreak 

\end{document}